\documentclass[12pt,reqno]{article}
\usepackage{amsmath,amsthm,amssymb, latexsym}
\usepackage{color}
\theoremstyle{plain}
\newtheorem{lem}{Lemma}[section]
\newtheorem{thm}[lem]{Theorem}
\newtheorem{prop}[lem]{Proposition}
\newtheorem{cor}[lem]{Corollary}
\theoremstyle{definition}

\newtheorem{defn}[lem]{Definition}

\numberwithin{equation}{section} \numberwithin{figure}{section}
\numberwithin{table}{section} 

\begin{document}

\title{A Forbidden-Minor Characterization for the Class of Graphic Matroids which yield the Co-graphic Element-Splitting Matroids}

%\author{S. B. DHOTRE$^{1}$, M. M. SHIKARE$^{2}$ AND P. P.MALAVADKAR$^{3}$}

\maketitle 
\begin{center}
S. B. DHOTRE\\
Department of Mathematics\\ University of Pune, Pune-411007 (India)\\
E-mail: dsantosh2@yahoo.co.in.
\vskip.1in

P. P. MALAVADKAR\\ 
MIT World Peace University, Pune-411038 (India)\\ 
E-mail: prashant.malavadkar@mitwpu.edu.in
\vskip.1in

\end{center}

\date{}

\maketitle
\begin{abstract}
  The element splitting operation on a graphic matroid, in general may not yield a cographic matroid. In this paper, we give a necessary and sufficient condition for the graphic matroid  to yield cographic matroid under the element splitting operation.
\end{abstract}



\normalsize
\section{Introduction} 
Fleischner \cite{FS} introduced the idea of splitting a vertex of degree at least three in a connected graph and used the operation to characterize Eulerian graphs. Raghunathan, Shikare and Waphare  \cite{RSWS}  extended the splitting operation from graphs to binary matroids. $M_{x,y}$ denotes the splitting matroid obtained by applying splitting operation on a binary matroid $M$, by a pair of elements $\{x, y\}$ of $M$.

Slater \cite{SlaterC4CG} specified the {\it n-point splitting operation} on a graph in the following way:\\
Let $G$ be a graph and $u$ be a vertex of degree at least $2n-2$ in
$G.$ Let $H$ be the graph obtained from $G$ by replacing $u$ by
two adjacent vertices $u_{1},~u_{2}$ such that each point formerly
joined to $u$ is joined to exactly one of $u_{1}$ and $u_{2}$ so
that in $H,$ deg$(u_{1})\geq n$ and deg$(u_{2})\geq n.$ We say that
$H$ arises from $G$ by $n$-element splitting operation. 

\vskip.2in
\unitlength=0.8mm \special{em:linewidth 0.4pt} \linethickness{0.4pt}
\begin{picture}(140.67,43.00)(20,0)
\put(35.00,20.00){\circle*{1.33}}
\put(70.00,20.00){\circle*{1.33}}
\put(70.00,42.33){\circle*{1.33}}
\put(35.00,42.33){\circle*{1.33}}
\put(52.67,30.67){\circle*{1.33}}
%\put(52.67,15.00){\circle*{1.33}}
\put(35.00,35.00){\circle*{1.33}}
\put(70.00,35.00){\circle*{1.33}}
\put(52.67,30.67){\line(-3,2){17.67}}
\put(52.67,30.67){\line(-4,1){17.67}}
\put(52.67,30.67){\line(-5,-3){17.67}}
%\put(52.67,15.00){\line(0,1){15.67}}
\put(52.67,30.67){\line(3,2){17.33}}
\put(52.67,30.33){\line(4,1){17.33}}
\put(52.67,30.67){\line(5,-3){17.33}}
\put(52.67,34.33){\makebox(0,0)[cc]{$u$}}
%\put(52.67,11.67){\makebox(0,0)[cc]{$v$}}
\put(30.67,42.67){\makebox(0,0)[cc]{$x_{1}$}}
\put(30.67,34.67){\makebox(0,0)[cc]{$x_{2}$}}
\put(30.67,19.33){\makebox(0,0)[cc]{$x_{k}$}}
\put(74.33,19.33){\makebox(0,0)[cc]{$y_{k}$}}
\put(74.33,34.67){\makebox(0,0)[cc]{$y_{2}$}}
\put(74.33,43.00){\makebox(0,0)[cc]{$y_{1}$}}
%\put(50.33,22.67){\makebox(0,0)[cc]{$e$}}
\put(38.33,31.67){\circle*{0.67}}
\put(38.33,28.67){\circle*{0.67}}
\put(38.33,25.67){\circle*{0.67}}
\put(65.33,31.00){\circle*{0.67}}
\put(65.33,28.33){\circle*{0.67}}
\put(65.33,26.00){\circle*{0.67}}
\put(90.00,20.00){\circle*{1.33}}
\put(136.67,20.00){\circle*{1.33}}
\put(136.67,35.00){\circle*{1.33}}
\put(136.67,42.33){\circle*{1.33}}
\put(90.00,42.33){\circle*{1.33}}
\put(90.00,35.00){\circle*{1.33}}
%\put(113.67,15.00){\circle*{1.33}}
\put(102.00,35.00){\circle*{1.33}}
\put(125.33,35.00){\circle*{1.33}}
\put(90.00,35.00){\line(1,0){46.67}}
\put(125.33,35.00){\line(3,2){11.33}}
\put(125.33,35.00){\line(3,-4){11.33}}
%\put(125.33,35.00){\line(-3,-5){12.00}}
%\put(113.67,15.00){\line(-3,5){12.00}}
\put(102.00,35.00){\line(-5,3){12.00}}
\put(102.00,35.00){\line(-4,-5){12.00}}
\put(102.00,39.00){\makebox(0,0)[cc]{$u_{1}$}}
\put(125.33,39.00){\makebox(0,0)[cc]{$u_{2}$}}
\put(85.67,42.33){\makebox(0,0)[cc]{$x_{1}$}}
\put(85.67,35.00){\makebox(0,0)[cc]{$x_{2}$}}
\put(85.67,20.00){\makebox(0,0)[cc]{$x_{k}$}}
%\put(113.67,12.33){\makebox(0,0)[cc]{$v$}}
\put(140.67,42.33){\makebox(0,0)[cc]{$y_{1}$}}
\put(140.67,35.00){\makebox(0,0)[cc]{$y_{2}$}}
\put(140.67,20.00){\makebox(0,0)[cc]{$y_{h}$}}
\put(113.67,37.67){\makebox(0,0)[cc]{$a$}}
%\put(105.00,24.33){\makebox(0,0)[cc]{$e$}}
%\put(122.00,24.33){\makebox(0,0)[cc]{$\gamma$}}
\put(93.67,32.67){\circle*{0.67}}
\put(93.67,30.33){\circle*{0.67}}
\put(93.67,28.00){\circle*{0.67}}
\put(132.67,33.33){\circle*{0.67}}
\put(132.67,31.33){\circle*{0.67}}
\put(132.67,29.00){\circle*{0.67}}
\put(52.67,10.67){\makebox(0,0)[cc]{$G$}}
\put(113.67,10.67){\makebox(0,0)[cc]{$H$}}
\put(79.67,3.67){\makebox(0,0)[cc]{\bf {Figure 1}}}
\end{picture}

If $ X = \{x_1, x_2,...x_k\}$ be the set of edges incident at $u_1$ then we denote $H$ by $G'_X$. Tutte \cite{TutteT3CG} characterized 3-connected graphs in terms of edge addition and 3-point splitting. Slater \cite{SlaterC4CG} obtained the following two useful results in this regard.
\begin{thm} The class of 2-connected graphs is the class of graphs obtained from $K_3$ by finite sequence of edge addition and 2-element splitting. \end{thm}

\begin{thm} If $G$ is $n$-connected and $H$ arise from $G$ by $n$-element splitting, then $H$ is $n$-connected.\end{thm}

Further, he classified 4-connected graphs using $n$-element splitting operation (see \cite{SlaterC4CG}).

Shikare and Azadi \cite{SeltBM, Azadi} extended the notion of $n$-point splitting operation on graphs to binary matroids as follows.
\begin{defn} Let $M$ be a binary matroid on a set $E$ and $A$ be a matrix over $GF(2)$ that represents the matroid $M.$
Suppose that $X$ is a subset of $E(M)$. Let
$A'_X$ be the matrix that is obtained by adjoining an extra
row to $A$ with this row being zero everywhere except in the
columns corresponding to the elements of $X$ where it takes the value $1$
and then adjoining an extra column (corresponding to $a$) with
this column being zero everywhere except in the last row where it
takes the value 1. Suppose $M'_X$ be the vector matroid of the matrix $A'_{X}.$ 
The transition from $M$ to $M'_X$ is called the element splitting operation.\end{defn} 

We call matroid $M'_X$ as the element splitting matroid. If $|X| = 2$ and $X = \{x, y\}$. We denote the matroid $M'_X$ by $M'_{x, y}$. Azadi \cite{Azadi}, characterized circuits of the element splitting matroid in terms of circuits of the binary matroids as follows.

\begin{prop} Let $M(E, \cal C)$ be a binary matroid together with the collection of circuits $\cal C$. Suppose $X\subseteq E$ and $a \notin E$. Then $M'_X = (E \cup \{a\}, \cal C')$ where $\cal C' = \cal {C}$$_0$ $ \cup$ $ \cal C$$_1$ $ \cup $ $\cal C$$_2$ and 
\begin{description}
\item $\cal C$$_0$ = $\{ C\in $ $\cal C$$ ~|~ C$ contains an even number of elements of $X$ $\}$;
\item $\cal C$$_1$ = The set of minimal members of $\{ C_1\cup C_2 ~|~ C_1, C_2\in $ $ \cal C$$, C_1\cap C_2 \neq \phi$ and each of $C_1$ and $C_2$ contains an odd number of elements of $X$ such that $C_1\cup C_2$ contains no member of $\cal C$$_0$ $\}$;
\item $\cal C$$_2$ = $\{C\cup \{a\}~|~C\in$ $\cal C$ and $C$ contains an odd number of elements of X $\}$;
\end{description}\end{prop}

Various properties concerning the element splitting matroids have been studied in \cite{Azadi, Habib, SeltBM}.

The element splitting operation on a graphic (cographic) matroid may not yield
a graphic (cographic) matroid. 

Dalvi, Borse and Shikare \cite{DBSGEltG, DBSCEltC} characterized  graphic (cographic) matroids whose element splitting matroids are graphic (cographic) when $|X| = 2$. In fact, they proved the following result.

\begin{thm}\label{gceltgc} The element splitting operation, by any pair of elements, on a graphic (cographic) matroid yields a graphic (cographic) matroid if and only if it has no minor isomorphic to $M(K_{4}),$ where $K_{4}$ is the complete graph on 4 vertices.\end{thm}

In this paper, we obtain a forbidden-minor characterization for graphic matroids whose element splitting matroid is cographic when $|X| = 2$. The main result in this paper is the following theorem. 

\begin{thm}\label{gcmain}{
The element splitting operation, by any pair of elements, on a graphic matroid yields a cographic matroid if and only if it has no minor isomorphic to $M(K_{4}),$ where $K_{4}$ is the complete graph on 4 vertices.}\end{thm}

\section{Properties of Element Splitting Operation } 
\noindent In this section, we provide necessary Lemmas which are used in the proof of Theorem \ref{gcmain}. Dalvi, Borse and Shikare \cite{DBSGEltG} proved the following useful Lemma. 

\begin{lem}\label{9 properties} { Let $x$ and $y$ be distinct elements of a binary matroid $M$ and let $r(M)$ denote the rank of $M.$ Then, using the notations introduced in Section 1,
\begin{description}
\item (i) $M_{x,y}=M'_{x,y}\setminus\{a\};$
\item (ii) $M=M'_{x,y}/\{a\};$
\item (iii) $r(M'_{x,y})=r(M)+1;$
\item (iv) Every cocircuit of $M$ is a cocircuit of the matroid $M'_{x,y};$
$(v)~$ if $\{x,y\}$ is a cocircuit of $M$ then $\{a\}$ and $\{x,y\}$ are cocircuits
of $M'_{x,y};$
\item (vi) If $\{x,y\}$ does not contain a cocircuit, then
$\{x,y,a\}$ is a cocircuit of $M'_{x,y};$
\item (vii) $M'_{x,y}\setminus x/y\cong M\setminus x;$
\item (viii) If $M$ is graphic and $x,y$ are adjacent edges in a
corresponding graph, then $M'_{x,y}$ is graphic;
\item (ix) $M'_{x,y}$ is not eulerian.\end{description} }\end{lem}

The following two results are well known minor based characterizations of graphic and cographic matroids (see \cite{Oxley} ).

\begin{thm}\label{Gminors}  A binary matroid is graphic if and only if it has no minor isomorphic to $F_{7},F_{7}^{*},M^{*}(K_{3,3})$ or
$M^{*}(K_{5})$. \end{thm}

\begin{thm}\label{Cminors} A binary matroid is cographic if and only if it has no minor isomorphic to $F_{7},F_{7}^{*},M^(K_{3,3})$ or
$M^(K_{5})$.\end{thm}

\noindent {\bf Notation}. For the sake of convenience, let \\ $\mathcal{F}=\{F_{7},~F_{7}^{*},~M(K_{5}),~M(K_{3,3})\}.$

In the following Lemma, we provide a necessary condition for a graphic matroid whose element splitting matroid is not cographic.
\begin{lem}\label{GEminors} Let $M$ be a graphic matroid and let
$x,y\in E(M)$ such that $M'_{x,y}$ is not cographic. Then $M$ has a minor isomorphic to $M(K_4)$ or there is
a minor $N$ of $M$ such that no two elements of $N$ are in series
and $N'_{x,y}\setminus \{a\}/\{x\}\cong F$ or $N'_{x,y}\setminus
\{a\}/\{x,y\}\cong F$ or $N'_{x,y}\cong F$ or $N'_{x,y}/\{x\}\cong
F$ or $N'_{x,y}/\{y\}\cong F$ or $N'_{x,y}/\{x,y\}\cong F$ for
some $F\in\mathcal{F}.$\end{lem}
\begin{proof} Suppose that $M'_{x,y}$ is not cographic and $M$ has no minor isomorphic to $M(K_4)$.
Since $M'_{x,y}$ is not cographic $M'_{x,y}\setminus T_1/T_2\cong F$ for some $T_1, T_2\subseteq
E(M'_{x,y}).$ Let $ T_i' = T_i -\{a, x, y\}$ for $i = 1, 2.$ Then
$ T_i' \subseteq E(M)$  for each $i.$ Let $ N = M \setminus
T_1'/T_2'.$ Then $N'_{x, y} = M'_{x,y}\setminus T_1'/T_2'.$ Let $
T_i'' = T_i - T_i'$ for $ i = 1, 2.$ Then $ N'_{x, y}\setminus
T_1''/T_2''\cong F.$ If $ a \in T_2'',$ then $F$ is a minor of $
M'_{x, y}/a $ and hence, by Lemma \ref{9 properties}$(i),~F$ is a minor of $M$. Since $M$ is graphic $F_7, F_7^*$ can not be the minors of $M$. So $F=M(K_5) $ or $F=M(K_{3,3}) $, but both $M(K_5) $ and $M(K_{3,3})$ have minor isomorphic to $M(K_4)$. Consequently $M$ has a minor isomorphic to $M(K_4)$. Which is a contradiction. Suppose
$ a \in T_1''.$ By Lemma \ref{9 properties} $(i),~M_{x, y} = M'_{x, y}\setminus a.$ Hence $F$ is a minor of
$ M_{x, y}.$  It follows from Theorem 2.3 of \cite{SWGSG} that $N$ does
not contain a $2$-cocircuit and further, $N_{x,y}/x\cong F$ or
$N_{x,y}/\{x,y\}\cong F.$ This implies that $N'_{x,y}\setminus
\{a\}/x\cong F$ or $N'_{x,y}\setminus \{a\}/\{x,y\}\cong F.$
Suppose that $ a \notin T_1'' \cup T_2''.$ Hence $ a \notin
T_1\cup T_2.$ If $ T_1'' \cup T_2'' = \phi,$ then $N'_{x,y}\cong
F.$ If $ T_2'' =\phi,$ then $ N_{x, y}\setminus x\cong F$ or
$N_{x,y}\setminus y \cong F $ or $ N'_{x, y} \setminus\{x, y\}
\cong F.$ In the first case, $ a$ forms a $2$-cocircuit with $x$
or $ y$ which ever is remained, and in the later case, $a$ is a coloop, both are contradictions. Hence $ T_2'' \neq \phi.$ If
$T_1''\neq\phi$ then, by Lemma \ref{9 properties}$(vi),~F$ is minor of $ M,$
which is a contradiction. Hence $ T_1'' = \phi$ and $ N'_{x,
y}/x \cong F$ or $N'_{x, y}/y \cong F$ or $ N'_{x, y}/{x, y} \cong
F.$
 Assume that $N$ contains a $2$-cocircuit $Q.$ By Lemma \ref{9 properties}$(iv),~Q$
is a $2$-cocircuit in $ N'_{x, y}.$ Since $F$ is $3$-connected, it
does not contain a $2$-cocircuit. It follows that $N'_{x,y}$ is
not isomorphic to $ F.$ Hence $N'_{x,y}\setminus\{a\}/x\cong F$ or
$N'_{x,y}\setminus \{a\}/\{x,y\}\cong F$ or $N'_{x,y}/\{x\}\cong
F$ or $N'_{x,y}/\{y\}\cong F$ or $N'_{x,y}/\{x,y\}\cong F.$ If $
Q\cap\{x, y\}= \phi,$ then it is retained in all these cases  and
thus $F$ has a $2$-cocircuit, which is a contradiction. If $Q=
\{x,y\},$ a contradiction follows from Lemma \ref{9 properties}$(v).$ Hence  $Q$
contains exactly one of $ x$ and $y$. Suppose that $x\in Q.$ Then
$N'_{x,y}/y\not\cong F.$ Let $x_1$ be the other element of $Q.$
Let $L =N/x_1.$ Then $L$ is a minor of $M$ in which no pair of
elements is in series. Further, $L'_{x,y}=N'_{x,y}/x_1\cong
N'_{x,y}/x.$ Thus we have $L'_{x,y}\setminus \{a\}\cong F$ or
$L'_{x,y}\setminus \{a\}/y\cong F$ or $L'_{x,y}\cong F$ or
$L'_{x,y}/y \cong F.$ Since $L_{x,y}\cong L'_{x,y}\setminus\{a\},$
and $x,y$ are in series in $L_{x,y},$ it follows that
$L'_{x,y}\setminus\{a\}\not\cong F$ and also
$L'_{x,y}\setminus\{a\}/y\cong L'_{x,y}\setminus\{a\}/x.$ If $y\in
Q,$ then $N'_{x,y}/x\not\cong F.$ Also, $L'_{x,y}\cong
N'_{x,y}/y.$ In this case we get $L'_{x,y}\setminus\{a\}/x\cong F$
or $L'_{x,y}\cong F$ or $L'_{x,y}/x\cong F.$\end{proof}

\begin{defn} Let $M$ be a graphic matroid in which no two elements are in series and let $F\in\mathcal{F}$. We say that $M$ is minimal with respect to $F$ and the element splitting operation if there exist two elements $x$ and $y$ of $M$ such that $M'_{x,y}\setminus \{a\}/\{x\}\cong F$ or $M'_{x,y}\setminus \{a\}/\{x,y\}\cong F$ or $M'_{x,y}\cong F$ or $M'_{x,y}/\{x\}\cong F$ or $M'_{x,y}/\{x,y\}\cong F.$\end{defn}

\begin{cor} \label{minimal}Let $M$ be a graphic matroid. For $x,y\in E(M),$ the matroid $M'_{x,y}$ is cographic if and only if $M$ has no minor isomorphic to a
minimal matroid with respect to any $F\in\mathcal {F}.$\end{cor}
\begin{proof} If $M'_{x,y}$ is not cographic for some $x,y$ then, by Lemma
\ref{GEminors}, $M$ has a minor $N$ in which no two elements are in series
and $N'_{x,y}\setminus \{a\}/\{x\}\cong F$ or $N'_{x,y}\setminus
\{a\}/\{x,y\}\cong F$ or $N'_{x,y}\cong F$ or $N'_{x,y}/\{x\}\cong
F$ or $N'_{x,y}/\{y\}\cong F$ or $N'_{x,y}/\{x,y\}\cong F$ for
some $F\in\mathcal{F}.$ If $N'_{x,y}/y\cong F$ but
$N'_{x,y}/x\not\cong F,$ then interchange roles of $x$ and $y.$

	Conversely, suppose that $M$ has a minor $N$ isomorphic to a
minimal matroid with respect to some $F\in\mathcal{F}.$ Then
$N'_{x,y}\setminus \{a\}$ or $N'_{x,y}/\{x\}$ or
$N'_{x,y}/\{x,y\}$ or $N'_{x,y}\cong F,~$ for some $x,y\in E(M).$
We conclude that $M'_{x,y}$ has a minor isomorphic to $F$ and hence it is not cographic.\end{proof}

In the following Lemma, we prove some basic properties of graphic minimal matroids.
\begin{lem}\label{8 properties}
Let $M$ be a graphic matroid. If $M$ is minimal with respect to some $F\in\mathcal{F},$ then\\
$(i)~~~M$ has neither loops nor coloops;\\
$(ii)~$ every pair of parallel elements of $M$ must contain either $x$ or $y;$\\
$(iii)~~x$ and $y$ cannot be parallel in $M;$\\
$(iv)~$ if $M'_{x,y}\cong F_{7}^{*}$ or $M(K_{3,3})$ then $M$ is
simple, and there is no odd \\$~~~~~~$circuit of $M$ containing
both $x$ and $y,$ and also there is no even\\$~~~~~$ circuit
of $M$ containing precisely one of $x$ and $y;$\\
$(v)~$ if $M'_{x,y}/\{x\}\cong F_{7}^{*}$ or $M(K_{3,3})$ then $M$
is simple and there
is no\\ $~~~~~~$ 3-circuit of $M$ containing both $x$ and $y;$\\
$(vi)$ if $M'_{x,y}/\{x\}\cong F_{7}$ or $M(K_{5}),$ then $M$ has
exactly one pair of\\ $~~~~~~~$parallel elements and there is no
3-circuit of
$M$ containing both \\$~~~~~~~~x$ and $y;$\\
$(vii)~$ if $M'_{x,y}/\{x,y\}\cong F$ then $M$ is simple and there
is no 3 or 4-circuit\\ $~~~~~~~$ of $M$ containing both $x$ and $y;$ and\\
$(viii)~~M'_{x,y}$ is not isomorphic to $F_{7}$ or
$M(K_{5}).$\end{lem}
\begin{proof} The proof is straightforward. \end{proof}

\begin{lem}\label{pirouz lemma} Let $F\in\mathcal{F}$ and let $M$ be a binary matroid such that
 either $M_{x,y}/\{x\}\cong F$ or $M_{x,y}/\{x,y\}\cong
F$ for some pair $x,y \in E(M).$ Then the following statements hold.\\
(i) $M$ has neither loops nor coloops;\\
(ii) $x$ and $y$ can not be parallel in $M$;\\
(iii) if $x_1$ and $x_2$ are parallel elements of $M$ , then one of them is either $~~~~~~~x$ or $y$ ;\\
(iv) if  $M_{x,y}/\{x,y\}\cong F,$ then $M$ has at most one pair of parallel\\ ~$~~~~~$elements;\\
(v) if $M_{x,y}/\{x\}\cong M(K_{3,3})$ or $M_{x,y}/\{x,y\}\cong M(K_{3,3}),$  then every odd\\$~~~~~$ circuit
    of $M$ contains $x$ or $y$; and \\
(vi) if $M_{x,y}/\{x\}\cong M(K_{5})$ or $M_{x,y}/\{x,y\}\cong M(K_{5}),$ then every odd\\$~~~~~$ cocircuit of
$M$ contains $x$ or $y.$\end{lem}

A matroid is said to be Eulerian if its ground set can be expressed as a union of circuits \cite{welsh}.

\begin{lem}\label{pirouz lemma1} \cite{SWGSG} Suppose $x$ and $y$ are non adjacent edges
 of a graph $G$ and $M=M(G)$. If $M_{x,y}/\{x,y\}$ is Eulerian, then either $G$ is Eulerian or
 the end vertices of $x$ and $ y$ are precisely the vertices of odd degree.\end{lem}

\section {A Forbidden-Minor Characterization for the Class of Graphic Matroids which yield the Cographic Element-Splitting Matroids} 
\noindent In this section, we obtain the minimal cographic matroids corresponding to each of the four matroids $F_{7},F_{7}^{*},M(K_{3,3})$ and $M(K_{5})$ and use
them to give a proof of Theorem \ref{gcmain}. \vskip.2cm \noindent

The minimal graphic matroids corresponding to the matroid $F_{7}$ and $F^*_{7}$ are characterized by Shikare and Dalvi \cite{DBSGEltG} in the following Lemma \ref{3.1} and Lemma \ref {3.2}

\begin{lem} \label{3.1} {Let $M$ be a graphic matroid. Then $M$ is minimal with respect to the matroid $F_{7}$ if and only if $M$ is isomorphic to one of
the cycle matroids $M(G_{1}),~M(G_{2})$ or $M(G_{3}),$ where $G_{1},~G_{2}$ and
$G_{3}$ are the graphs of figure 2.} \end{lem}

\unitlength=0.8mm \special{em:linewidth 0.4pt} \linethickness{0.4pt}
\begin{picture}(105.00,38.67)
\put(40.00,15.00){\circle*{1.33}}
\put(54.67,15.00){\circle*{1.33}}
\put(54.67,26.33){\circle*{1.33}}
\put(40.00,26.33){\circle*{1.33}}
\put(64.67,26.33){\circle*{1.33}}
\put(78.67,26.33){\circle*{1.33}}
\put(78.67,15.00){\circle*{1.33}}
\put(64.67,15.00){\circle*{1.33}}
\put(88.00,15.00){\circle*{1.33}}
\put(103.00,15.00){\circle*{1.33}}
\put(103.00,26.33){\circle*{1.33}}
\put(88.00,26.33){\circle*{1.33}}
\put(71.67,35.33){\circle*{1.33}}
\put(40.00,26.33){\line(1,0){14.67}}
\put(54.67,26.33){\line(0,-1){11.33}}
\put(54.67,15.00){\line(-1,0){14.67}}
\put(40.00,15.00){\line(0,1){11.33}}
\put(40.00,26.33){\line(4,-3){14.67}}
\put(40.00,15.00){\line(4,3){14.67}}
\put(64.67,26.33){\line(4,5){7.33}}
\put(71.67,35.33){\line(4,-5){7.33}}
\put(78.67,26.33){\line(0,-1){11.33}}
\put(78.67,15.00){\line(-1,0){14.00}}
\put(64.67,15.00){\line(0,1){11.33}}
\put(64.67,26.33){\line(1,0){14.00}}
\put(78.67,15.00){\line(-1,3){6.67}}
\put(64.67,26.33){\line(5,-4){14.00}}
\put(88.00,26.33){\line(1,0){15.00}}
\put(103.00,26.33){\line(0,-1){11.33}}
\put(103.00,15.00){\line(-1,0){15.00}}
\put(88.00,15.00){\line(0,1){11.33}}
\put(88.00,26.33){\line(4,-3){15.00}}
\put(88.00,15.00){\line(4,3){15.00}}
%\put(37.67,20.67){\makebox(0,0)[cc]{$1$}}
\put(47.33,24.50){\makebox(0,0)[cc]{$y$}}
\put(47.33,16.33){\makebox(0,0)[cc]{$x$}}
%\put(42.97,22.00){\makebox(0,0)[cc]{$4$}}
%\put(53.00,22.00){\makebox(0,0)[cc]{$5$}}
%\put(56.67,20.67){\makebox(0,0)[cc]{$z$}}
%\put(66.67,32.67){\makebox(0,0)[cc]{$5$}}
\put(77.33,32.00){\makebox(0,0)[cc]{$y$}}
%\put(80.67,20.67){\makebox(0,0)[cc]{$w$}}
\put(71.67,16.67){\makebox(0,0)[cc]{$x$}}
\put(62.67,20.67){\makebox(0,0)[cc]{$u$}}
%\put(67.67,21.33){\makebox(0,0)[cc]{$2$}}
%\put(74.33,22.33){\makebox(0,0)[cc]{$3$}}
%\put(70.33,28.33){\makebox(0,0)[cc]{$4$}}
\put(95.33,28.33){\makebox(0,0)[cc]{$y$}}
\put(95.33,16.67){\makebox(0,0)[cc]{$x$}}
%\put(105.00,20.67){\makebox(0,0)[cc]{$1$}}
%\put(86.50,20.67){\makebox(0,0)[cc]{$2$}}
%\put(100.33,22.00){\makebox(0,0)[cc]{$3$}}
%\put(90.67,22.00){\makebox(0,0)[cc]{$4$}}
\bezier{48}(40.00,26.33)(44.00,32.67)(48.33,32.00)
\bezier{104}(40.00,26.33)(47.67,37.33)(54.67,26.33)
\bezier{96}(40.00,15.00)(47.33,5.67)(54.67,15.00)
\bezier{96}(64.67,15.00)(71.67,5.00)(78.67,15.00)
\bezier{116}(88.00,26.33)(95.33,38.67)(103.00,26.33)
\put(47.33,33.67){\makebox(0,0)[cc]{$u$}}
\put(47.33,8.33){\makebox(0,0)[cc]{$v$}}
\put(71.67,8.00){\makebox(0,0)[cc]{$v$}}
\put(95.33,34.67){\makebox(0,0)[cc]{$u$}}
\put(47.33,3.33){\makebox(0,0)[cc]{$G_{1}$}}
\put(72.67,3.33){\makebox(0,0)[cc]{$G_{2}$}}
\put(95.67,3.33){\makebox(0,0)[cc]{$G_{3}$}}
\put(69.67,-4.00){\makebox(0,0)[cc]{\bf Figure 2}}
\end{picture}\\

\begin{lem}\label{3.2} { Let $M$ be a graphic matroid. Then $M$ is minimal
with respect to the matroid $F_{7}^{*}$ if and only if $M$ is
isomorphic the cycle matroid $M(G_{4})$ or
$M(G_{5}),$ where $G_{4}$ and $G_{5}$ are the graphs of figure 3.} \hfill $\square$ \end{lem}

\unitlength=0.8mm \special{em:linewidth 0.4pt} \linethickness{0.4pt}
\begin{picture}(87.67,32.67)
\put(50.00,15.00){\circle*{1.33}}
\put(62.67,15.00){\circle*{1.33}}
\put(62.67,25.67){\circle*{1.33}}
\put(50.00,25.67){\circle*{1.33}}
\put(56.67,32.00){\circle*{1.33}}
\put(73.33,25.67){\circle*{1.33}}
\put(86.00,25.67){\circle*{1.33}}
\put(86.00,15.00){\circle*{1.33}}
\put(73.33,15.00){\circle*{1.33}}
\put(50.00,15.00){\line(1,0){12.67}}
\put(62.67,15.00){\line(0,1){10.67}}
\put(62.67,25.67){\line(-1,1){6.00}}
\put(56.67,32.00){\line(-1,-1){6.67}}
\put(50.00,25.67){\line(0,-1){10.67}}
\put(50.00,25.67){\line(1,0){12.67}}
\put(56.67,32.00){\line(-2,-5){6.67}}
\put(73.33,15.00){\line(1,0){12.67}}
\put(86.00,15.00){\line(0,1){10.67}}
\put(86.00,25.67){\line(-1,0){12.67}}
\put(73.33,25.67){\line(0,-1){10.67}}
\put(73.33,15.00){\line(6,5){12.67}}
\put(73.33,25.67){\line(6,-5){12.67}}
\bezier{92}(50.00,15.00)(55.67,5.67)(62.67,15.00)
\put(56.67,8.00){\makebox(0,0)[cc]{$$}}
\put(53.50,18.67){\makebox(0,0)[cc]{$$}}
\put(58.67,23.50){\makebox(0,0)[cc]{$$}}
\put(61.33,30.00){\makebox(0,0)[cc]{$$}}
\put(64.33,20.67){\makebox(0,0)[cc]{$$}}
\put(48.33,20.67){\makebox(0,0)[cc]{$$}}
\put(56.67,16.67){\makebox(0,0)[cc]{$x$}}
\put(56.67,12.00){\makebox(0,0)[cc]{$u$}}
\put(52.00,30.00){\makebox(0,0)[cc]{$y$}}
\put(71.67,19.33){\makebox(0,0)[cc]{$$}}
\put(87.67,20.33){\makebox(0,0)[cc]{$$}}
\put(75.67,21.67){\makebox(0,0)[cc]{$$}}
\put(84.00,22.00){\makebox(0,0)[cc]{$$}}
\put(79.67,27.33){\makebox(0,0)[cc]{$x$}}
\put(79.67,12.00){\makebox(0,0)[cc]{$y$}}
\put(68.67,-4.00){\makebox(0,0)[cc]{\bf Figure 3}}
\put(56.67,3.00){\makebox(0,0)[cc]{$G_{4}$}}
\put(79.67,3.00){\makebox(0,0)[cc]{$G_{5}$}}
\end{picture}\\

In the following Lemma, the minimal matroids corresponding to the
matroid $M(K_{3,3})$ are characterized. \begin{lem} \label{3.3} Let $M$ be a graphic matroid. Then $M$ is minimal with respect to the matroid $M(K_{3,3})$ if and only if $M$ is isomorphic to $M(G_{6}),$ $M(G_{7}),$ $M(G_{8}),$ $M(G_{9})$, or $~M(G_{10})$, where $G_{6},$ $G_{7},$ $G_{8},$ $G_{9},$ $G_{10}$ are the graphs of Figure 4.\end{lem}
%TeXCAD Picture [Figure 4.117.pic]. Options:
%\grade{\on}
%\emlines{\off}
%\epic{\off}
%\beziermacro{\on}
%\reduce{\on}
%\snapping{\off}
%\quality{8.000}
%\graddiff{0.005}
%\snapasp{1}
%\zoom{4.0000}
\unitlength 0.8mm % = 2.845pt
\linethickness{0.4pt}
\ifx\plotpoint\undefined\newsavebox{\plotpoint}\fi % GNUPLOT compatibility
\begin{picture}(165.375,87.33)(15,0)
\put(59.33,12.5){\circle*{1.33}}
\put(77.08,12.75){\circle*{1.33}}
\put(86.17,24.17){\circle*{1.33}}
\put(47.25,24.08){\circle*{1.33}}
\put(59.58,35.42){\circle*{1.33}}
\put(76.58,35.08){\circle*{1.33}}
\put(68.58,23.75){\circle*{1.33}}
\put(59.58,35.42){\line(1,0){17}}
\put(76.58,35.42){\line(0,0){0}}
\put(76.58,35.42){\line(5,-6){9.33}}
\put(85.92,24.08){\line(0,0){0}}
\put(76.58,12.42){\line(0,0){0}}
\put(59.25,12.42){\line(-1,1){12}}
\put(47.25,24.42){\line(0,0){0}}
\put(47,25.17){\line(6,5){12.33}}
\put(59.58,34.75){\line(0,0){0}}
\put(59.58,34.75){\line(5,-6){9}}
\put(68.58,23.75){\line(0,0){0}}
\put(68.58,23.75){\line(1,0){17.33}}
\put(68.58,24.17){\line(-1,0){21}}
\put(59.25,12.75){\line(5,6){9}}
\put(57.75,22.17){\makebox(0,0)[cc]{$x$}}
\put(107.25,12.42){\circle*{1.33}}
\put(125.92,12.08){\circle*{1.33}}
\put(125.92,28.75){\circle*{1.33}}
\put(116.58,36.08){\circle*{1.33}}
\put(107.25,28.42){\circle*{1.33}}
\put(107.25,29.08){\line(1,0){18.67}}
\put(125.92,29.08){\line(0,0){0}}
\put(125.92,29.08){\line(0,-1){17}}
\put(125.92,12.08){\line(0,0){0}}
\put(125.92,12.08){\line(-1,0){18.33}}
\put(107.58,12.08){\line(0,0){0}}
\put(107.58,12.08){\line(0,1){16.67}}
\put(107.58,28.75){\line(0,0){0}}
\put(107.58,28.75){\line(6,5){9}}
\put(116.58,36.08){\line(0,0){0}}
\put(116.58,36.08){\line(5,-4){9}}
\put(107.58,12.08){\line(2,5){9.67}}
\put(117.25,36.42){\line(0,0){0}}
\put(117.25,36.42){\line(1,-3){8}}
\put(110.58,35.42){\makebox(0,0)[cc]{$x$}}
\put(119.92,21.08){\makebox(0,0)[cc]{$y$}}
\put(65.5,6.58){\makebox(0,0)[cc]{$G_9$}}
\put(80.5,2.92){\makebox(0,0)[cc]{\bf Figure 4}}
\put(17,54.33){\circle*{1.33}}
\put(131.5,56.33){\circle*{1.33}}
\put(74.5,55.33){\circle*{1.33}}
\put(17,76.67){\circle*{1.33}}
\put(131.5,78.67){\circle*{1.33}}
\put(74.5,77.67){\circle*{1.33}}
\put(37.67,76.33){\circle*{1.33}}
\put(152.17,78.33){\circle*{1.33}}
\put(95.17,77.33){\circle*{1.33}}
\put(37.67,54){\circle*{1.33}}
\put(152.17,56){\circle*{1.33}}
\put(95.17,55){\circle*{1.33}}
\put(50,64.67){\circle*{1.33}}
\put(164.5,66.67){\circle*{1.33}}
\put(107.5,65.67){\circle*{1.33}}
\put(4.67,64.33){\circle*{1.33}}
\put(119.17,66.33){\circle*{1.33}}
\put(62.17,65.33){\circle*{1.33}}
\put(4.67,65){\line(1,1){12.33}}
\put(119.17,67){\line(1,1){12.33}}
\put(62.17,66){\line(1,1){12.33}}
\put(17,77.33){\line(1,0){20.67}}
\put(131.5,79.33){\line(1,0){20.67}}
\put(74.5,78.33){\line(1,0){20.67}}
\put(37.67,77.33){\line(1,-1){12.33}}
\put(95.17,78.33){\line(1,-1){12.33}}
\put(50,65){\line(-6,-5){12.67}}
\put(164.5,67){\line(-6,-5){12.67}}
\put(107.5,66){\line(-6,-5){12.67}}
\put(37.33,54.33){\line(-1,0){20.33}}
\put(151.83,56.33){\line(-1,0){20.33}}
\put(94.83,55.33){\line(-1,0){20.33}}
\put(17,54.33){\line(-5,4){12.33}}
\put(131.5,56.33){\line(-5,4){12.33}}
\put(74.5,55.33){\line(-5,4){12.33}}
\put(4.67,64.33){\line(1,0){45.33}}
\put(119.92,66.83){\line(1,0){45.33}}
\put(62.17,65.33){\line(1,0){45.33}}
\put(131.17,78){\line(0,-1){21.67}}
\put(74.17,77){\line(0,-1){21.67}}
\bezier{112}(74.17,78)(83.5,87.33)(95.17,78)
\put(40.75,68.33){\makebox(0,0)[cc]{$x$}}
\put(156.5,70.08){\makebox(0,0)[cc]{$x$}}
\put(84.5,75.33){\makebox(0,0)[cc]{$x$}}
\put(7.08,72.83){\makebox(0,0)[cc]{$y$}}
\put(132.83,61.33){\makebox(0,0)[cc]{$y$}}
\put(75.83,60.33){\makebox(0,0)[cc]{$y$}}
\put(25.67,50.33){\makebox(0,0)[cc]{$G_6$}}
\put(140.17,52.33){\makebox(0,0)[cc]{$G_8$}}
\put(83.17,51.33){\makebox(0,0)[cc]{$G_7$}}
%\emline(152.5,78.5)(164.75,67)
\multiput(152.5,78.5)(.0359237537,-.0337243402){341}{\line(1,0){.0359237537}}
%\end
\qbezier(152.5,79)(165.375,78.25)(164.75,66.5)
%\emline(37.75,76.5)(37.5,54.25)
\multiput(37.75,76.5)(-.03125,-2.78125){8}{\line(0,-1){2.78125}}
%\end
%\emline(16.75,54.25)(38,77)
\multiput(16.75,54.25)(.0337301587,.0361111111){630}{\line(0,1){.0361111111}}
%\end
\put(74.25,29.25){\makebox(0,0)[cc]{$y$}}
\put(114,8.5){\makebox(0,0)[cc]{$G_{10}$}}
%\emline(17,76.75)(37.25,54.5)
\multiput(17,76.75)(.0336938436,-.0370216306){601}{\line(0,-1){.0370216306}}
%\end
%\emline(131.5,78.75)(152.5,55.75)
\multiput(131.5,78.75)(.0337078652,-.036918138){623}{\line(0,-1){.036918138}}
%\end
%\emline(74.5,77.75)(95,55.25)
\multiput(74.5,77.75)(.0337171053,-.0370065789){608}{\line(0,-1){.0370065789}}
%\end
%\emline(76.5,35)(77,13.25)
\multiput(76.5,35)(.0333333,-1.45){15}{\line(0,-1){1.45}}
%\end
%\emline(77,13.25)(86.25,23.75)
\multiput(77,13.25)(.0336363636,.0381818182){275}{\line(0,1){.0381818182}}
%\end
\put(77,12.5){\line(-1,0){18.25}}
\end{picture}\vskip.2cm\noindent
\begin{proof} We have
$M'(G_{6})_{x,y}\setminus\{a\}/\{x\}$ $\cong
M(K_{3,3}),$ $M'(G_{7})_{x,y}\setminus\{a\}/\{x\}$ $\cong
M(K_{3,3}),$ $M'(G_{8})_{x,y}\setminus\{a\}/\{x\}$ $\cong
M(K_{3,3}),$ $M'(G_{9})_{x,y}\setminus\{a\}/\{x,y\}$ $\cong
M(K_{3,3}),$ $M'(G_{10})_{x,y}$ $\cong M(K_{3,3}).$ Therefore
$M(G_{6}),$ $M(G_{7}),$ $M(G_{8}),$ $M(G_{9}),$ $M(G_{10})$ are minimal
matroids with respect to the matroid
$M(K_{3,3}).$\vskip.2cm\noindent Conversely, suppose that $M$ is a
minimal matroid with respect to the matroid $M(K_{3,3}).$ Then
there exist elements  $x$ and $y$ of $M$ such that
$M'_{x,y}\setminus\{a\}/\{x\}\cong M(K_{3,3})$ or
$M'_{x,y}\setminus\{a\}/\{x,y\}\cong M(K_{3,3})$ or $M'_{x,y}\cong
M(K_{3,3})$ or $M'_{x,y}/\{x\}\cong M(K_{3,3})$ or
$M'_{x,y}/\{x,y\}\cong M(K_{3,3})$ and further, no two elements of
$M$ are in series. \vskip.2cm\noindent {\bf Case (i)}
$M'_{x,y}\setminus\{a\}/\{x\}\cong M(K_{3,3}).$\vskip.2cm\noindent
By Lemmas \ref{pirouz lemma}, \ref{pirouz lemma1} and \ref{8
properties}, $M_{x,y}/\{x\}\cong M(K_{3,3})$. Since
$r(M_{x,y}/\{x\})=r(M(K_{3,3})=5$. $M_{x,y}$ is a matroid of rank
6 and $\left|E(M)\right|=10$. In the light of the Lemma \ref{9
properties}(iii), the matroid $M$  has rank 5 and its ground set
has 10 elements. Let $G$ be a connected graph corresponding to
$M$.Then $G$ has 6 vertices, 10 edges, and has no
 vertex of degree 2. Hence, by Lemma \ref{8 properties}, $G$ has minimum degree at least 3 since no two elements
 are in series.Thus the degree sequence of $G$ is (5,3,3,3,3,3,3) or (4,4,3,3,3,3). By Harary [\cite{Harary},
 p 223], each simple connected graph with these degree sequences is isomorphic to one of the graphs of Figure 5
 below.\\
 
 %TeXCAD Picture [Copy of FIG7.bak]. Options:
%\grade{\on}
%\emlines{\off}
%\epic{\off}
%\beziermacro{\on}
%\reduce{\on}
%\snapping{\off}
%\quality{8.000}
%\graddiff{0.005}
%\snapasp{1}
%\zoom{4.0000}
\unitlength 0.8mm % = 2.845pt
\linethickness{0.4pt}
\ifx\plotpoint\undefined\newsavebox{\plotpoint}\fi % GNUPLOT compatibility
\begin{picture}(117.495,75.995)(0,0)
\put(34.83,52.25){\circle*{1.33}}
\put(85.33,52.75){\circle*{1.33}}
\put(32.33,18.25){\circle*{1.33}}
\put(85.58,19.5){\circle*{1.33}}
\put(35.08,73.83){\circle*{1.33}}
\put(85.33,75.33){\circle*{1.33}}
\put(32.58,39.83){\circle*{1.33}}
\put(85.83,41.08){\circle*{1.33}}
\put(54.75,74.25){\circle*{1.33}}
\put(105.25,74.75){\circle*{1.33}}
\put(52.25,40.25){\circle*{1.33}}
\put(105.5,41.5){\circle*{1.33}}
\put(54.5,51.92){\circle*{1.33}}
\put(105,52.42){\circle*{1.33}}
\put(52,17.92){\circle*{1.33}}
\put(105.25,19.17){\circle*{1.33}}
\put(66.08,62.5){\circle*{1.33}}
\put(116.58,63){\circle*{1.33}}
\put(63.58,28.5){\circle*{1.33}}
\put(116.83,29.75){\circle*{1.33}}
\put(22.75,62.17){\circle*{1.33}}
\put(73.25,62.67){\circle*{1.33}}
\put(20.25,28.17){\circle*{1.33}}
\put(73.5,29.42){\circle*{1.33}}
\put(22.75,62.5){\line(1,1){12.33}}
\put(73.25,63){\line(1,1){12.33}}
\put(20.25,28.5){\line(1,1){12.33}}
\put(73.5,29.75){\line(1,1){12.33}}
\put(35.08,74.83){\line(0,0){0}}
\put(85.58,75.33){\line(0,0){0}}
\put(32.58,40.83){\line(0,0){0}}
\put(85.83,42.08){\line(0,0){0}}
\put(35.08,74.83){\line(1,0){19.67}}
\put(85.58,75.33){\line(1,0){19.67}}
\put(32.58,40.83){\line(1,0){19.67}}
\put(85.83,42.08){\line(1,0){19.67}}
\put(54.75,74.83){\line(0,0){0}}
\put(105.25,75.33){\line(0,0){0}}
\put(52.25,40.83){\line(0,0){0}}
\put(105.5,42.08){\line(0,0){0}}
\put(54.75,52.17){\line(0,0){0}}
\put(105.25,52.67){\line(0,0){0}}
\put(52.25,18.17){\line(0,0){0}}
\put(105.5,19.42){\line(0,0){0}}
\put(54.5,51.92){\line(-1,0){19.67}}
\put(105,52.42){\line(-1,0){19.67}}
\put(52,17.92){\line(-1,0){19.67}}
\put(105.25,19.17){\line(-1,0){19.67}}
\put(35.08,52.17){\line(0,0){0}}
\put(85.58,52.67){\line(0,0){0}}
\put(32.58,18.17){\line(0,0){0}}
\put(85.83,19.42){\line(0,0){0}}
\put(35.08,52.17){\line(-5,4){12.33}}
\put(85.58,52.67){\line(-5,4){12.33}}
\put(32.58,18.17){\line(-5,4){12.33}}
\put(85.83,19.42){\line(-5,4){12.33}}
\put(22.75,62.17){\line(0,0){0}}
\put(73.25,62.67){\line(0,0){0}}
\put(20.25,28.17){\line(0,0){0}}
\put(73.5,29.42){\line(0,0){0}}
\put(22.75,62.17){\line(3,-1){32}}
\put(73.25,62.67){\line(3,-1){32}}
\put(20.25,28.17){\line(3,-1){32}}
\put(54.75,51.5){\line(0,0){0}}
\put(105.25,52){\line(0,0){0}}
\put(52.25,17.5){\line(0,0){0}}
\put(105.5,18.75){\line(0,0){0}}
\put(54.5,51.25){\line(1,1){11.33}}
\put(105,51.75){\line(1,1){11.33}}
\put(52,17.25){\line(1,1){11.33}}
\put(105.25,18.5){\line(1,1){11.33}}
\put(66.08,62.83){\line(0,0){0}}
\put(116.58,63.33){\line(0,0){0}}
\put(63.58,28.83){\line(0,0){0}}
\put(116.83,30.08){\line(0,0){0}}
\put(66.08,62.83){\line(-1,1){11.33}}
\put(116.58,63.33){\line(-1,1){11.33}}
\put(63.58,28.83){\line(-1,1){11.33}}
\put(116.83,30.08){\line(-1,1){11.33}}
\put(54.75,74.17){\line(0,0){0}}
\put(105.25,74.67){\line(0,0){0}}
\put(52.25,40.17){\line(0,0){0}}
\put(105.5,41.42){\line(0,0){0}}
%\emline(34.5,52.75)(65.75,62.5)
\multiput(34.5,52.75)(.1081314879,.0337370242){289}{\line(1,0){.1081314879}}
%\end
%\emline(53.75,52.75)(54.5,74.5)
\multiput(53.75,52.75)(.0326087,.9456522){23}{\line(0,1){.9456522}}
%\end
%\emline(35,74.25)(54.25,51.5)
\multiput(35,74.25)(.0337127846,-.0398423818){571}{\line(0,-1){.0398423818}}
%\end
%\emline(73.25,63)(105,75)
\multiput(73.25,63)(.0891853933,.0337078652){356}{\line(1,0){.0891853933}}
%\end
%\emline(85.5,74.75)(116.25,63.25)
\multiput(85.5,74.75)(.0901759531,-.0337243402){341}{\line(1,0){.0901759531}}
%\end
%\emline(116.25,63.25)(84.5,52.5)
\multiput(116.25,63.25)(-.0995297806,-.0336990596){319}{\line(-1,0){.0995297806}}
%\end
%\emline(51.75,40)(52,18)
\multiput(51.75,40)(.03125,-2.75){8}{\line(0,-1){2.75}}
%\end
%\emline(32,18.5)(63.5,28.75)
\multiput(32,18.5)(.1036184211,.0337171053){304}{\line(1,0){.1036184211}}
%\end
%\emline(32,40.5)(63.25,28.75)
\multiput(32,40.5)(.0895415473,-.0336676218){349}{\line(1,0){.0895415473}}
%\end
%\emline(73.25,29.25)(116.5,30.25)
\multiput(73.25,29.25)(1.4416667,.0333333){30}{\line(1,0){1.4416667}}
%\end
%\emline(105.25,40.75)(105,19.5)
\multiput(105.25,40.75)(-.03125,-2.65625){8}{\line(0,-1){2.65625}}
%\end
%\emline(105,19.5)(85.25,41.5)
\multiput(105,19.5)(-.0337030717,.0375426621){586}{\line(0,1){.0375426621}}
%\end
%\emline(105.5,42)(85,19.25)
\multiput(105.5,42)(-.0337171053,-.0374177632){608}{\line(0,-1){.0374177632}}
%\end
\put(43.25,45.5){\makebox(0,0)[cc]{$i$}}
\put(95.25,47.5){\makebox(0,0)[cc]{$ii$}}
\put(42.25,13.5){\makebox(0,0)[cc]{$iii$}}
\put(95,15){\makebox(0,0)[cc]{$iv$}}
\put(68,8.25){\makebox(0,0)[cc]{\bf Figure 5}}
\end{picture}

By the nature of the circuits of $M(K_{3,3}) $ or $M_{x,y}$ and by Lemma \ref{8 properties},
it follows that $G$ can not have,
(i) two or more edge disjoint triangles,
(ii) a circuit of size 3 or 4 or 6 containing both $x$ and $y$. Since each of
the graphs $(i), (ii)$ and $(iii)$ of Figure 5 contains two or more edge disjoint triangles,
we discard them.The graph $(iv)$ of Figure 5, is isomorphic to the graph $G_6$ in the statement of the Lemma.\\
   Suppose $G$ is a multigraph. Then by Lemma 3.3 of \cite{BSDCSC},
   $G$ is isomorphic  to $G_7$ or $G_8$ of Figure 4.
 \vskip.2cm\noindent {\bf Case (ii)}.
$M'_{x,y}\setminus\{a\}/\{x,y\}\cong
M(K_{3,3}).$\vskip.2cm\noindent By Lemma \ref{9 properties}$(i)$, $M_{x,y}/\{x,y\}\cong M(K_{3,3}).$
 As $r(M_{x,y})=7$ $r(M(K_{3,3})=5$. Hence $r(M) =6$ and $\left|E(M)\right|=11$. Let $G$ be connected
 graph corresponding to $M$. Then $G$ has 7 vertices, 11 edges and has minimum degree at least 3.
  Therefore the degree sequence of $G$ is (4,3,3,3,3,3,3). It follows from Lemma \ref{8 properties}
   that $G$ can not have \\
(i) more than two edge disjoint triangles;
(ii) a cycle of size other than 6 which contains both $x$ and $y$; and
(iii) a triangle and a 2-circuit which are edge disjoint.\\
Then, by case (ii) of Lemma 3.3 of \cite{BSDCSC}, $G$ is isomorphic to $G_9$ of Figure 4.
\\
\vskip.2cm\noindent {\bf Case (iii)}. $M'_{x,y}\cong
M(K_{3,3}).$\vskip.2cm\noindent Since $r(M(K_{3,3}))=5$ and
$|E(M(K_{3,3}))|=9,~r(M)=4$ and $|E(M)|=8.$ Therefore $M$ cannot
have $M(K_{3,3})$ and $M(K_{5})$ as a minor. We conclude that $M$ is
cographic. Let $G$ be a graph which corresponds to the matroid $M.$
Then $G$ has 5 vertices and 8 edges. As $M$ is graphic and
cographic, $G$ is planar. By Lemma \ref{8 properties}$(iv),~G$ is simple. Since
$M$ has no coloop and no two elements are in series, minimum
degree in $G$ is at least 3. There is only one non-isomorphic
simple graph with 5 vertices and 8
edges \cite{Harary}, see Figure 6.\\

\unitlength=0.8mm \special{em:linewidth 0.4pt} \linethickness{0.4pt}
\begin{picture}(40.67,30.67)
\put(65.00,10.00){\circle*{1.33}}
\put(80.00,10.00){\circle*{1.33}}
\put(80.00,22.00){\circle*{1.33}}
\put(65.00,22.00){\circle*{1.33}}
\put(72.67,30.00){\circle*{1.33}}
\put(65.00,22.00){\line(1,1){7.67}}
\put(72.67,30.00){\line(1,-1){7.33}}
\put(80.00,22.00){\line(0,-1){12.00}}
\put(80.00,10.00){\line(-1,0){15.00}}
\put(65.00,10.00){\line(0,1){12.00}}
\put(65.00,22.00){\line(1,0){15.00}}
\put(80.00,10.00){\line(-2,5){8.00}}
\put(72.67,30.00){\line(-2,-5){8.00}}
\put(72.67,4.67){\makebox(0,0)[cc]{\bf Figure 6 : Simple graph with 5 vertices and 8 edges}}
\end{picture}
\vskip.2cm\noindent Thus, $G$ is isomorphic to this
graph, which is nothing but the graph $G_{10}$ of Figure 4.

\vskip.2cm\noindent {\bf Case (iv)}. $M'_{x,y}/\{x\}\cong
M(K_{3,3}).$\vskip.2cm\noindent As in the above cases, considering the
 rank of $M(K_{3,3})$ we have $r(M)=5$ and $|E(M)|=9.$
If $M$ is not cographic then $M$ has minor a isomorphic to $M(K_{3,3})$ or $M(K_{5})$.
Suppose $M$ has a minor isomorphic to $M(K_{3,3})$ and $M'_{x,y}/\{x\}\cong M(K_{3,3}).$
Then $M\cong M(K_{3,3}).$\\
Now $x,y\in M$ and since $M$ is isomorphic to $ M(K_{3,3})$, $x, y$ lie in a 4-circuit say $C$
 but then $C-\{x\}$ is a triangle in $M'_{x,y}/\{x\}\cong M(K_{3,3})$, a contradiction
 to the fact that $M'_{x,y}/\{x\}$ is bipartite and does not contain any odd circuit. Thus $M$ does
  not contain $ M(K_{3,3})$ as a minor.\\
If $M$ has minor a isomorphic to $M(K_{5})$ then $r(M)=5$, $|E(M)|=9.$ and $|E(M(K_{5})|=10$, so $M$
does not contain $M(K_{5})$ as a minor. Hence $M$ is cographic.\\
Let $M=M(G)$ be graphic matroid. Then
$G$ has 6 vertices and 9 edges. Further, $G$ is simple and planar.
Also, minimum degree in $G$ is
at least 3. Thus $G$ is isomorphic to the graph of \\\\
\unitlength=0.8mm \special{em:linewidth 0.4pt} \linethickness{0.4pt}
\begin{picture}(83.00,26.00)
\put(65.00,10.00){\circle*{1.33}}
\put(77.00,10.00){\circle*{1.33}}
\put(59.67,17.67){\circle*{1.33}}
\put(65.00,25.33){\circle*{1.33}}
\put(77.00,25.33){\circle*{1.33}}
\put(82.33,17.67){\circle*{1.33}}
\put(59.67,17.67){\line(2,3){5.00}}
\put(65.00,25.33){\line(1,0){12.00}}
\put(77.00,25.33){\line(2,-3){5.00}}
\put(82.33,17.67){\line(-2,-3){5.00}}
\put(77.00,10.00){\line(-1,0){12.00}}
\put(65.00,10.00){\line(-2,3){5.00}}
\put(59.67,17.67){\line(1,0){22.67}}
\put(77.00,25.33){\line(0,-1){15.33}}
\put(65.00,10.00){\line(0,1){15.33}}
\put(71.00,3.00){\makebox(0,0)[cc]{\bf Figure 7}}
\end{picture}
\vskip.2cm\noindent Figure 7 (see \cite{Harary}). If a triangle of $G$
contains neither  $x$ nor $y$ then it is preserved in
$M'_{x,y}/\{x\},$ a contradiction. Hence $x$ belongs to a triangle
of $G.$ This gives rise to a 4-circuit in $M'_{x,y}$ containing
$x$ and $a.$ Hence, we get a 3-circuit in $M'_{x,y}/\{x\},$ a
contradiction. Thus $M$ is not isomorphic to the graph in Figure 7.\vskip.2cm\noindent {\bf Case (v)}.
$M'_{x,y}/\{x,y\}\cong M(K_{3,3}).$\vskip.2cm\noindent Then
$r(M)=6$ and $|E(M)|=10.$ If $M$ is not cographic then, $M$ has minor isomorphic to $M(K_{3,3})$ or $M(K_{5})$.\\
Suppose $M$ has a minor isomorphic to $M(K_{3,3})$ then, since $M$ is graphic, $M$ must
 be the graph in Figure 8 (see \cite{Harary}).\\
 
%TeXCAD Picture [FIG10.PIC]. Options:
%\grade{\on}
%\emlines{\off}
%\epic{\off}
%\beziermacro{\on}
%\reduce{\on}
%\snapping{\off}
%\quality{8.000}
%\graddiff{0.005}
%\snapasp{1}
%\zoom{4.0000}
\unitlength 0.8mm % = 2.845pt
\linethickness{0.4pt}
\ifx\plotpoint\undefined\newsavebox{\plotpoint}\fi % GNUPLOT compatibility
\begin{picture}(51.335,37.335)(-55.4,3)
\put(14,14.33){\circle*{1.33}}
\put(14,36.67){\circle*{1.33}}
\put(31.67,36){\circle*{1.33}}
\put(50.67,36){\circle*{1.33}}
\put(50.67,14){\circle*{1.33}}
\put(31.67,14){\circle*{1.33}}
\put(14,24.33){\circle*{1.33}}
\put(32.25,36.25){\line(-1,0){.75}}
\put(31.5,36.25){\line(0,1){0}}
%\emline(14,36.5)(13.75,24.25)
\multiput(14,36.5)(-.03125,-1.53125){8}{\line(0,-1){1.53125}}
%\end
\put(13.75,24.25){\line(0,-1){9.5}}
%\emline(13.75,14.75)(32,36)
\multiput(13.75,14.75)(.0337338262,.0392791128){541}{\line(0,1){.0392791128}}
%\end
%\emline(32,36)(50.75,14.25)
\multiput(32,36)(.0337230216,-.039118705){556}{\line(0,-1){.039118705}}
%\end
\put(50.75,14.25){\line(0,1){22}}
%\emline(50.75,36.25)(31.5,14.25)
\multiput(50.75,36.25)(-.0337127846,-.0385288967){571}{\line(0,-1){.0385288967}}
%\end
%\emline(31.5,14.25)(14,36.5)
\multiput(31.5,14.25)(-.0337186898,.0428709056){519}{\line(0,1){.0428709056}}
%\end
%\emline(14,36.5)(50.75,14)
\multiput(14,36.5)(.05509745127,-.03373313343){667}{\line(1,0){.05509745127}}
%\end
%\emline(31.75,35.75)(31.25,14.25)
\multiput(31.75,35.75)(-.0333333,-1.4333333){15}{\line(0,-1){1.4333333}}
%\end
%\emline(13.75,14.25)(50.25,35.75)
\multiput(13.75,14.25)(.0572100313,.0336990596){638}{\line(1,0){.0572100313}}
%\end
\put(28.25,5.5){\makebox(0,0)[cc]{\bf Figure 8}}
\end{picture}

Any two edges in the graph of Figure 8 are in a 4-cycle or in a 5-cycle so are $x$
and $y$ also. Then these circuits ( cycles in $G$ ) will be preserved in $M'_{x,y}$
and hence a 2-circuit or a  triangle is formed in $M'_{x,y}/\{x,y\}\cong M(K_{3,3})$,
a contradiction to the fact that $M'_{x,y}/\{x,y\}\cong M(K_{3,3})$ is a simple bipartite matroid.
Thus $M$ has no minor isomorphic to $M(K_{3,3})$.\\
Suppose that $M$ has a minor isomorphic to $M(K_{5})$ but then
$|E(M)|=10 $, hence $M= M(K_{5})$. Consequently $r(M)=4$ and this is a contradiction to the fact that $r(M)=6$. Thus $M$ does not have a
minor isomorphic to $M(K_{5})$. Thus $M$ is cographic.
We conclude that $M$ is graphic as well as cographic. Suppose $G$ is a
graph corresponding to $M$, then $G$ has  7 vertices and 10 edges.
This implies that $G$ has at least one vertex of degree 2, which is
a contradiction to the fact  that $M$ has no 2-cocircuit.
Therefore, the situation $M'_{x,y}/\{x,y\}\cong M(K_{3,3})$ does not
occur.\end{proof}

Finally, we characterize minimal matroids corresponding to the
matroid $M(K_{5})$ in the following Lemma.

\vskip.2cm\noindent \begin{lem} \label{3.4} Let $M$ be a graphic matroid. Then $M$ is minimal with respect to the matroid $M(K_{5})$ if and only if $M$
is isomorphic to one of the seven matroids
$M(G_{11}),~M(G_{12}),~M(G_{13}),~M(G_{14}),~M(G_{15}),~M(G_{16})$ and
$M(G_{17}),$ where $G_{11},~G_{12},~G_{13},~G_{14},~G_{15},~G_{16}$ and
$G_{17}$ are the graphs of Figure 9.\end{lem}

%TeXCAD Picture [Figure 9 .pic.bak]. Options:
%\grade{\on}
%\emlines{\off}
%\epic{\off}
%\beziermacro{\on}
%\reduce{\on}
%\snapping{\off}
%\quality{8.000}
%\graddiff{0.005}
%\snapasp{1}
%\zoom{4.0000}
\unitlength 0.8mm % = 2.845pt
\linethickness{0.4pt}
\ifx\plotpoint\undefined\newsavebox{\plotpoint}\fi % GNUPLOT compatibility
\begin{picture}(167.67,95.625)(10,0)
\put(17,58.92){\circle*{1.33}}
\put(109,18.92){\circle*{1.33}}
\put(35.67,58.58){\circle*{1.33}}
\put(127.67,18.58){\circle*{1.33}}
\put(35.67,75.25){\circle*{1.33}}
\put(127.67,35.25){\circle*{1.33}}
\put(26.33,82.58){\circle*{1.33}}
\put(118.33,42.58){\circle*{1.33}}
\put(132.08,40.83){\circle*{1.33}}
\put(147.58,40.83){\circle*{1.33}}
\put(165.08,40.33){\circle*{1.33}}
\put(165.33,19.33){\circle*{1.33}}
\put(149.08,19.08){\circle*{1.33}}
\put(132.58,19.33){\circle*{1.33}}
\put(17,74.92){\circle*{1.33}}
\put(109,34.92){\circle*{1.33}}
\put(35.67,75.58){\line(0,0){0}}
\put(127.67,35.58){\line(0,0){0}}
\put(35.67,75.58){\line(0,-1){17}}
\put(127.67,35.58){\line(0,-1){17}}
\put(35.67,58.58){\line(0,0){0}}
\put(127.67,18.58){\line(0,0){0}}
\put(35.67,58.58){\line(-1,0){18.33}}
\put(127.67,18.58){\line(-1,0){18.33}}
\put(17.33,58.58){\line(0,0){0}}
\put(109.33,18.58){\line(0,0){0}}
\put(17.33,58.58){\line(0,1){16.67}}
\put(109.33,18.58){\line(0,1){16.67}}
\put(17.33,75.25){\line(0,0){0}}
\put(109.33,35.25){\line(0,0){0}}
\put(17.33,75.25){\line(6,5){9}}
\put(109.33,35.25){\line(6,5){9}}
\put(26.33,82.58){\line(0,0){0}}
\put(118.33,42.58){\line(0,0){0}}
\put(17.33,58.58){\line(2,5){9.67}}
\put(109.33,18.58){\line(2,5){9.67}}
\put(27,82.92){\line(0,0){0}}
\put(119,42.92){\line(0,0){0}}
\put(27,82.92){\line(1,-3){8}}
\put(119,42.92){\line(1,-3){8}}
\put(19.08,80.67){\makebox(0,0)[cc]{$x$}}
\put(32.17,80.58){\makebox(0,0)[cc]{$y$}}
\put(90.42,62.58){\makebox(0,0)[cc]{$y$}}
\put(148.67,60.33){\makebox(0,0)[cc]{$y$}}
\put(8.92,29.83){\makebox(0,0)[cc]{$y$}}
\put(64.92,31.08){\makebox(0,0)[cc]{$y$}}
\put(110.42,39.33){\makebox(0,0)[cc]{$y$}}
\put(167.67,29.08){\makebox(0,0)[cc]{$y$}}
\put(62.25,59.33){\circle*{1.33}}
\put(121.75,58.08){\circle*{1.33}}
\put(14.5,18.83){\circle*{1.33}}
\put(71,19.83){\circle*{1.33}}
\put(62.25,81.67){\circle*{1.33}}
\put(121.75,80.42){\circle*{1.33}}
\put(14.5,41.17){\circle*{1.33}}
\put(71,42.17){\circle*{1.33}}
\put(82.92,81.33){\circle*{1.33}}
\put(142.42,80.08){\circle*{1.33}}
\put(35.17,40.83){\circle*{1.33}}
\put(91.67,41.83){\circle*{1.33}}
\put(82.92,59){\circle*{1.33}}
\put(142.42,57.75){\circle*{1.33}}
\put(35.17,18.5){\circle*{1.33}}
\put(91.67,19.5){\circle*{1.33}}
\put(95.25,69.67){\circle*{1.33}}
\put(154.75,68.42){\circle*{1.33}}
\put(47.5,29.17){\circle*{1.33}}
\put(104,30.17){\circle*{1.33}}
\put(49.92,69.33){\circle*{1.33}}
\put(109.42,68.08){\circle*{1.33}}
\put(2.17,28.83){\circle*{1.33}}
\put(58.67,29.83){\circle*{1.33}}
\put(49.92,70){\line(1,1){12.33}}
\put(109.42,68.75){\line(1,1){12.33}}
\put(2.17,29.5){\line(1,1){12.33}}
\put(58.67,30.5){\line(1,1){12.33}}
\put(14.5,41.83){\line(1,0){20.67}}
\put(71,42.83){\line(1,0){20.67}}
\put(82.92,82.33){\line(1,-1){12.33}}
\put(142.42,81.08){\line(1,-1){12.33}}
\put(35.17,41.83){\line(1,-1){12.33}}
\put(91.67,42.83){\line(1,-1){12.33}}
\put(95.25,70){\line(-6,-5){12.67}}
\put(154.75,68.75){\line(-6,-5){12.67}}
\put(47.5,29.5){\line(-6,-5){12.67}}
\put(104,30.5){\line(-6,-5){12.67}}
\put(82.58,59.33){\line(-1,0){20.33}}
\put(142.08,58.08){\line(-1,0){20.33}}
\put(34.83,18.83){\line(-1,0){20.33}}
\put(91.33,19.83){\line(-1,0){20.33}}
\put(62.25,59.33){\line(-5,4){12.33}}
\put(121.75,58.08){\line(-5,4){12.33}}
\put(14.5,18.83){\line(-5,4){12.33}}
\put(71,19.83){\line(-5,4){12.33}}
\put(49.92,69.33){\line(1,0){45.33}}
\put(109.42,68.08){\line(1,0){45.33}}
\put(2.17,28.83){\line(1,0){45.33}}
\put(58.67,29.83){\line(1,0){45.33}}
\put(14.17,40.5){\line(0,-1){21.67}}
\put(70.67,41.5){\line(0,-1){21.67}}
\put(91.25,76.33){\makebox(0,0)[cc]{$x$}}
\put(135.75,78.83){\makebox(0,0)[cc]{$x$}}
\put(125.5,38.58){\makebox(0,0)[cc]{$x$}}
\put(16,33.08){\makebox(0,0)[cc]{$x$}}
\put(72.5,33.58){\makebox(0,0)[cc]{$x$}}
\put(141.25,43.08){\makebox(0,0)[cc]{$x$}}
%\emline(62.25,81.75)(82.5,59.5)
\multiput(62.25,81.75)(.0336938436,-.0370216306){601}{\line(0,-1){.0370216306}}
%\end
%\emline(121.75,80.5)(142,58.25)
\multiput(121.75,80.5)(.0336938436,-.0370216306){601}{\line(0,-1){.0370216306}}
%\end
%\emline(14.5,41.25)(35,18.75)
\multiput(14.5,41.25)(.0337171053,-.0370065789){608}{\line(0,-1){.0370065789}}
%\end
%\emline(71,42.25)(91.5,19.75)
\multiput(71,42.25)(.0337171053,-.0370065789){608}{\line(0,-1){.0370065789}}
%\end
%\emline(17.5,74.5)(35.75,58.25)
\multiput(17.5,74.5)(.0378630705,-.0337136929){482}{\line(1,0){.0378630705}}
%\end
%\emline(109.5,34.5)(127.75,18.25)
\multiput(109.5,34.5)(.0378630705,-.0337136929){482}{\line(1,0){.0378630705}}
%\end
%\emline(35.75,75.5)(17.25,59)
\multiput(35.75,75.5)(-.037755102,-.0336734694){490}{\line(-1,0){.037755102}}
%\end
%\emline(127.75,35.5)(109.25,19)
\multiput(127.75,35.5)(-.037755102,-.0336734694){490}{\line(-1,0){.037755102}}
%\end
\qbezier(16.75,75)(13.125,87.625)(26,82.75)
\qbezier(108.75,35)(105.125,47.625)(118,42.75)
\qbezier(26.75,82.5)(37.875,87.375)(35.5,75.75)
%\emline(35.5,75.75)(26.25,82.75)
\multiput(35.5,75.75)(-.044471154,.033653846){208}{\line(-1,0){.044471154}}
%\end
%\emline(127.5,35.75)(118.25,42.75)
\multiput(127.5,35.75)(-.044471154,.033653846){208}{\line(-1,0){.044471154}}
%\end
\put(62.25,81.5){\line(0,-1){22.5}}
\put(121.75,80.25){\line(0,-1){22.5}}
%\emline(62.25,82)(95,69.75)
\multiput(62.25,82)(.0899725275,-.0336538462){364}{\line(1,0){.0899725275}}
%\end
%\emline(121.75,80.75)(154.5,68.5)
\multiput(121.75,80.75)(.0899725275,-.0336538462){364}{\line(1,0){.0899725275}}
%\end
%\emline(62.25,58.75)(94.75,69.5)
\multiput(62.25,58.75)(.1018808777,.0336990596){319}{\line(1,0){.1018808777}}
%\end
%\emline(121.75,57.5)(154.25,68.25)
\multiput(121.75,57.5)(.1018808777,.0336990596){319}{\line(1,0){.1018808777}}
%\end
%\emline(49.5,69.5)(83.25,81.5)
\multiput(49.5,69.5)(.0948033708,.0337078652){356}{\line(1,0){.0948033708}}
%\end
%\emline(109,68.25)(142.75,80.25)
\multiput(109,68.25)(.0948033708,.0337078652){356}{\line(1,0){.0948033708}}
%\end
\qbezier(83,81.5)(98,83.75)(95,70)
\qbezier(109.25,68.5)(108.125,95.625)(142.5,80.25)
\put(35,40.75){\line(0,-1){22}}
\put(91.5,41.75){\line(0,-1){22}}
%\emline(14.75,19)(35.25,41.25)
\multiput(14.75,19)(.0337171053,.0365953947){608}{\line(0,1){.0365953947}}
%\end
%\emline(71.25,20)(91.75,42.25)
\multiput(71.25,20)(.0337171053,.0365953947){608}{\line(0,1){.0365953947}}
%\end
\qbezier(1.75,28.75)(23.625,70.25)(47,29.75)
\qbezier(70.75,42)(34.25,31)(70.75,20)
%\emline(131.75,40.75)(132.25,19.5)
\multiput(131.75,40.75)(.0333333,-1.4166667){15}{\line(0,-1){1.4166667}}
%\end
%\emline(132.25,19.5)(147.5,41.25)
\multiput(132.25,19.5)(.0337389381,.048119469){452}{\line(0,1){.048119469}}
%\end
%\emline(147.5,41.25)(165,19.75)
\multiput(147.5,41.25)(.0337186898,-.0414258189){519}{\line(0,-1){.0414258189}}
%\end
\put(165,19.75){\line(0,1){20.5}}
%\emline(165,40.25)(148.75,19.25)
\multiput(165,40.25)(-.0337136929,-.0435684647){482}{\line(0,-1){.0435684647}}
%\end
%\emline(148.75,19.25)(131.75,41)
\multiput(148.75,19.25)(-.0337301587,.0431547619){504}{\line(0,1){.0431547619}}
%\end
%\emline(131.75,41)(165,19.5)
\multiput(131.75,41)(.0521159875,-.0336990596){638}{\line(1,0){.0521159875}}
%\end
%\emline(148,40.75)(147.5,39.25)
\multiput(148,40.75)(-.0333333,-.1){15}{\line(0,-1){.1}}
%\end
%\emline(147.5,40.75)(149,19.25)
\multiput(147.5,40.75)(.03333333,-.47777778){45}{\line(0,-1){.47777778}}
%\end
%\emline(132.25,19.75)(164.5,40.25)
\multiput(132.25,19.75)(.0530427632,.0337171053){608}{\line(1,0){.0530427632}}
%\end
%\emline(132.75,41)(147,41.25)
\multiput(132.75,41)(1.78125,.03125){8}{\line(1,0){1.78125}}
%\end
\put(25,53.5){\makebox(0,0)[cc]{$G_{11}$}}
\put(70.5,54.25){\makebox(0,0)[cc]{$G_{12}$}}
\put(129.5,54.25){\makebox(0,0)[cc]{$G_{13}$}}
\put(24.25,13){\makebox(0,0)[cc]{$G_{14}$}}
\put(77.75,14.5){\makebox(0,0)[cc]{$G_{15}$}}
\put(116.75,14.5){\makebox(0,0)[cc]{$G_{16}$}}
\put(146.25,14.5){\makebox(0,0)[cc]{$G_{17}$}}
\put(91.25,6){\makebox(0,0)[cc]{\bf Figure 9}}
%\emline(132.25,19.5)(148.75,19.25)
\multiput(132.25,19.5)(2.0625,-.03125){8}{\line(1,0){2.0625}}
%\end
\end{picture}

\begin{proof} We have
$M'(G_{11})_{x,y}\setminus\{a\}/\{x\}\cong M(K_{5}),$
$~M'(G_{12})_{x,y}\setminus\{a\}/\{x,y\}\cong M(K_{5}),$
$~M'(G_{13})_{x,y}\setminus\{a\}/\{x,y\}\cong M(K_{5}),$
$~M'(G_{16})_{x,y}/\{x\}\cong M(K_{5}),$ \\
$~M'(G_{14})_{x,y}\setminus\{a\}/\{x,y\}\cong M(K_{5}),$
$~M'(G_{15})_{x,y}\setminus\{a\}/\{x,y\}\cong M(K_{5})$\\
$~M'(G_{17})_{x,y}/\{x,y\}\cong M(K_{5}).$
 Therefore $M(G_{11}),~M(G_{12}),\\ ~M(G_{13}),~M(G_{14}),~M(G_{15}), ~M(G_{16})$ and
$M(G_{17})$ are minimal matroids with respect to the matroid
$M(K_{5}).$\vskip.2cm\noindent
Conversely, suppose that $M$ is a minimal matroid with respect to
the matroid $M(K_{5}).$ Then there exist elements $x$ and $y$ of
$M$ such that $M'_{x,y}\setminus\{a\}/\{x\}\cong M(K_{5})$ or
$M'_{x,y}\setminus\{a\}/\{x,y\}\cong M(K_{5})$ or $M'_{x,y}\cong
M(K_{5})$ or $M'_{x,y}/\{x\}\cong M(K_{5})$ or
$M'_{x,y}/\{x,y\}\cong M(K_{5})$ and also, $M$ does not contain a
2-cocircuit.\vskip.2cm\noindent {\bf Case (i)}.
$M'_{x,y}\setminus\{a\}/\{x\}\cong M(K_{5}).$\vskip.2cm\noindent
By Lemma \ref{9 properties}$(i),~M(G)_{x,y}/\{x\}\cong M(K_{5}).$ Hence, by similar argument as in Lemma 3.4, in \cite{DBSCEltC}, $M$ is isomorphic to the cycle matroid $M(G_{11})$, where $G_{11}$
is the graph of Figure 9.\\
\vskip.2cm\noindent {\bf Case (ii)}.
$M'_{x,y}\setminus\{a\}/\{x,y\}\cong M(K_{5}).$\vskip.2cm\noindent
By Lemma \ref{9 properties}$(i),$ $M(G)_{x,y}/\{x,y\}\cong M(K_{5}).$ Then
$r(M(K_{5}))= 4,~r(M_{x,y})=6$ and $|E(M)|=12.$  Let $G$ be a connected graph
corresponding to $M$.Then $G$ has 6 vertices, 12 edges and has minimum degree
at least 3. Suppose that $G$ is simple. By Lemma \ref{3.4} of \cite{BSDCSC}, there are 5 non
isomorphic simple graphs each with 6 vertices and 12 edges, out of which, two graphs are
 discarded in case $(ii)$ of Lemma \ref{3.4} of \cite{BSDCSC}. So, only three graphs are
 remaining and these graphs are not planar. These graphs are given in Figure 10.\\
 
%TeXCAD Picture [fig12.bak]. Options:
%\grade{\on}
%\emlines{\off}
%\epic{\off}
%\beziermacro{\on}
%\reduce{\on}
%\snapping{\off}
%\quality{8.000}
%\graddiff{0.005}
%\snapasp{1}
%\zoom{4.0000}
\unitlength 0.8mm % = 2.845pt
\linethickness{0.4pt}
\ifx\plotpoint\undefined\newsavebox{\plotpoint}\fi % GNUPLOT compatibility
\begin{picture}(194.75,39.83)(0,0)
\put(20.08,14.75){\circle*{1.33}}
\put(72.08,15){\circle*{1.33}}
\put(123.58,15.5){\circle*{1.33}}
\put(20.33,36.33){\circle*{1.33}}
\put(72.33,36.58){\circle*{1.33}}
\put(123.83,37.08){\circle*{1.33}}
\put(40,36.75){\circle*{1.33}}
\put(92,37){\circle*{1.33}}
\put(143.5,37.5){\circle*{1.33}}
\put(39.75,14.42){\circle*{1.33}}
\put(91.75,14.67){\circle*{1.33}}
\put(143.25,15.17){\circle*{1.33}}
\put(51.33,25){\circle*{1.33}}
\put(103.33,25.25){\circle*{1.33}}
\put(154.83,25.75){\circle*{1.33}}
\put(8,24.67){\circle*{1.33}}
\put(60,24.92){\circle*{1.33}}
\put(111.5,25.42){\circle*{1.33}}
\put(8,25){\line(1,1){12.33}}
\put(60,25.25){\line(1,1){12.33}}
\put(111.5,25.75){\line(1,1){12.33}}
\put(20.33,37.33){\line(0,0){0}}
\put(72.33,37.58){\line(0,0){0}}
\put(123.83,38.08){\line(0,0){0}}
\put(20.33,37.33){\line(1,0){19.67}}
\put(72.33,37.58){\line(1,0){19.67}}
\put(123.83,38.08){\line(1,0){19.67}}
\put(40,37.33){\line(0,0){0}}
\put(92,37.58){\line(0,0){0}}
\put(143.5,38.08){\line(0,0){0}}
\put(194.75,39.83){\line(0,0){0}}
\put(40,14.67){\line(0,0){0}}
\put(92,14.92){\line(0,0){0}}
\put(143.5,15.42){\line(0,0){0}}
\put(39.75,14.42){\line(-1,0){19.67}}
\put(91.75,14.67){\line(-1,0){19.67}}
\put(143.25,15.17){\line(-1,0){19.67}}
\put(20.33,14.67){\line(0,0){0}}
\put(72.33,14.92){\line(0,0){0}}
\put(123.83,15.42){\line(0,0){0}}
\put(20.33,14.67){\line(-5,4){12.33}}
\put(72.33,14.92){\line(-5,4){12.33}}
\put(123.83,15.42){\line(-5,4){12.33}}
\put(8,24.67){\line(0,0){0}}
\put(60,24.92){\line(0,0){0}}
\put(111.5,25.42){\line(0,0){0}}
\put(40,14){\line(0,0){0}}
\put(92,14.25){\line(0,0){0}}
\put(143.5,14.75){\line(0,0){0}}
\put(39.75,13.75){\line(1,1){11.33}}
\put(91.75,14){\line(1,1){11.33}}
\put(143.25,14.5){\line(1,1){11.33}}
\put(51.33,25.33){\line(0,0){0}}
\put(103.33,25.58){\line(0,0){0}}
\put(154.83,26.08){\line(0,0){0}}
\put(51.33,25.33){\line(-1,1){11.33}}
\put(103.33,25.58){\line(-1,1){11.33}}
\put(154.83,26.08){\line(-1,1){11.33}}
\put(40,36.67){\line(0,0){0}}
\put(92,36.92){\line(0,0){0}}
\put(143.5,37.42){\line(0,0){0}}
\put(28.5,8){\makebox(0,0)[cc]{$(i)$}}
%\emline(20.5,36.25)(51,25)
\multiput(20.5,36.25)(.0913173653,-.0336826347){334}{\line(1,0){.0913173653}}
%\end
%\emline(72.5,36.5)(103,25.25)
\multiput(72.5,36.5)(.0913173653,-.0336826347){334}{\line(1,0){.0913173653}}
%\end
%\emline(20.25,37)(39.75,14.75)
\multiput(20.25,37)(.0337370242,-.0384948097){578}{\line(0,-1){.0384948097}}
%\end
%\emline(72.25,37.25)(91.75,15)
\multiput(72.25,37.25)(.0337370242,-.0384948097){578}{\line(0,-1){.0384948097}}
%\end
%\emline(123.75,37.75)(143.25,15.5)
\multiput(123.75,37.75)(.0337370242,-.0384948097){578}{\line(0,-1){.0384948097}}
%\end
%\emline(20.25,36.75)(20,14.5)
\multiput(20.25,36.75)(-.03125,-2.78125){8}{\line(0,-1){2.78125}}
%\end
%\emline(72.25,37)(72,14.75)
\multiput(72.25,37)(-.03125,-2.78125){8}{\line(0,-1){2.78125}}
%\end
%\emline(39.75,37)(20,14.5)
\multiput(39.75,37)(-.0337030717,-.0383959044){586}{\line(0,-1){.0383959044}}
%\end
%\emline(143.25,37.75)(123.5,15.25)
\multiput(143.25,37.75)(-.0337030717,-.0383959044){586}{\line(0,-1){.0383959044}}
%\end
%\emline(50.75,25.25)(7.75,24.75)
\multiput(50.75,25.25)(-2.8666667,-.0333333){15}{\line(-1,0){2.8666667}}
%\end
%\emline(102.75,25.5)(59.75,25)
\multiput(102.75,25.5)(-2.8666667,-.0333333){15}{\line(-1,0){2.8666667}}
%\end
%\emline(154.25,26)(111.25,25.5)
\multiput(154.25,26)(-2.8666667,-.0333333){15}{\line(-1,0){2.8666667}}
%\end
%\emline(50.75,25.25)(19.75,14.5)
\multiput(50.75,25.25)(-.0971786834,-.0336990596){319}{\line(-1,0){.0971786834}}
%\end
%\emline(102.75,25.5)(71.75,14.75)
\multiput(102.75,25.5)(-.0971786834,-.0336990596){319}{\line(-1,0){.0971786834}}
%\end
%\emline(154.25,26)(123.25,15.25)
\multiput(154.25,26)(-.0971786834,-.0336990596){319}{\line(-1,0){.0971786834}}
%\end
%\emline(59.75,25)(92,14.75)
\multiput(59.75,25)(.1060855263,-.0337171053){304}{\line(1,0){.1060855263}}
%\end
\put(143.25,37.75){\line(0,-1){22.25}}
%\emline(143.25,15.5)(111,26)
\multiput(143.25,15.5)(-.1033653846,.0336538462){312}{\line(-1,0){.1033653846}}
%\end
\put(80,10.5){\makebox(0,0)[cc]{$(ii)$}}
\put(131.75,11.75){\makebox(0,0)[cc]{$(iii)$}}
\put(87.5,4.5){\makebox(0,0)[cc]{\bf Figure 10}}
\end{picture}
 In graph $(i)$ of Figure 10, not every odd cocircuit of $M$ contains $x$ or $y$,
  a contradiction to the fact that  if $M(G)_{x,y}/\{x,y\}\cong M(K_{5}),$ then every
   odd cocircuit contain $x$ or $y$ otherwise if both of them are absent then that odd
   cocircuit of $M$ is the odd cocircuit in $M(G)_{x,y}/\{x,y\}\cong M(K_{5})$. Consequently
   $M(G)_{x,y}/\{x,y\}(\cong M(K_{5}))$ becomes non Eulerian. In each of the graphs $(ii)$ and
    $(iii)$ of Figure 10, $x$ and $y$ together belong to a 3-cycle or a 4-cycle, a contradiction
    to Lemma \ref{pirouz lemma}\\
 Suppose that $G$ is not simple. Then, by Lemma \ref{pirouz lemma} $(iv)$,
 $G$ has exactly one pair of parallel edges. Then $G$ can be obtained from
 a simple graph on 6 vertices and 11 edges by adding a parallel edge.
 %TeXCAD Picture [fig 13.bak]. Options:
%\grade{\on}
%\emlines{\off}
%\epic{\off}
%\beziermacro{\on}
%\reduce{\on}
%\snapping{\off}
%\quality{8.000}
%\graddiff{0.005}
%\snapasp{1}
%\zoom{4.0000}
\unitlength 0.8mm % = 2.845pt
\linethickness{0.4pt}
\ifx\plotpoint\undefined\newsavebox{\plotpoint}\fi % GNUPLOT compatibility
\begin{picture}(141.495,105.495)(85,0)
\put(15.33,81.25){\circle*{1.33}}
\put(15.08,52.25){\circle*{1.33}}
\put(83.08,21.75){\circle*{1.33}}
\put(62.83,81.25){\circle*{1.33}}
\put(35.33,21){\circle*{1.33}}
\put(109.58,82.25){\circle*{1.33}}
\put(61.33,52){\circle*{1.33}}
\put(108.58,53){\circle*{1.33}}
\put(15.58,102.83){\circle*{1.33}}
\put(15.33,73.83){\circle*{1.33}}
\put(83.33,43.33){\circle*{1.33}}
\put(62.83,103.83){\circle*{1.33}}
\put(35.33,43.58){\circle*{1.33}}
\put(109.58,104.83){\circle*{1.33}}
\put(61.58,73.58){\circle*{1.33}}
\put(108.83,74.58){\circle*{1.33}}
\put(35.25,103.25){\circle*{1.33}}
\put(35,74.25){\circle*{1.33}}
\put(103,43.75){\circle*{1.33}}
\put(82.75,103.25){\circle*{1.33}}
\put(55.25,43){\circle*{1.33}}
\put(129.5,104.25){\circle*{1.33}}
\put(81.25,74){\circle*{1.33}}
\put(128.5,75){\circle*{1.33}}
\put(35,80.92){\circle*{1.33}}
\put(34.75,51.92){\circle*{1.33}}
\put(102.75,21.42){\circle*{1.33}}
\put(82.5,80.92){\circle*{1.33}}
\put(55,20.67){\circle*{1.33}}
\put(129.25,81.92){\circle*{1.33}}
\put(81,51.67){\circle*{1.33}}
\put(128.25,52.67){\circle*{1.33}}
\put(46.58,91.5){\circle*{1.33}}
\put(46.33,62.5){\circle*{1.33}}
\put(114.33,32){\circle*{1.33}}
\put(94.08,91.5){\circle*{1.33}}
\put(66.58,31.25){\circle*{1.33}}
\put(140.83,92.5){\circle*{1.33}}
\put(92.58,62.25){\circle*{1.33}}
\put(139.83,63.25){\circle*{1.33}}
\put(3.25,91.17){\circle*{1.33}}
\put(3,62.17){\circle*{1.33}}
\put(71,31.67){\circle*{1.33}}
\put(50.75,91.17){\circle*{1.33}}
\put(23.25,30.92){\circle*{1.33}}
\put(97.5,92.17){\circle*{1.33}}
\put(49.25,61.92){\circle*{1.33}}
\put(96.5,62.92){\circle*{1.33}}
\put(3.25,91.5){\line(1,1){12.33}}
\put(3,62.5){\line(1,1){12.33}}
\put(71,32){\line(1,1){12.33}}
\put(50.75,91.5){\line(1,1){12.33}}
\put(23.25,31.25){\line(1,1){12.33}}
\put(97.5,92.5){\line(1,1){12.33}}
\put(49.25,62.25){\line(1,1){12.33}}
\put(96.5,63.25){\line(1,1){12.33}}
\put(15.58,103.83){\line(0,0){0}}
\put(15.33,74.83){\line(0,0){0}}
\put(83.33,44.33){\line(0,0){0}}
\put(63.08,103.83){\line(0,0){0}}
\put(35.58,43.58){\line(0,0){0}}
\put(109.83,104.83){\line(0,0){0}}
\put(61.58,74.58){\line(0,0){0}}
\put(108.83,75.58){\line(0,0){0}}
\put(63.08,103.83){\line(1,0){19.67}}
\put(35.58,43.58){\line(1,0){19.67}}
\put(109.83,104.83){\line(1,0){19.67}}
\put(61.58,74.58){\line(1,0){19.67}}
\put(108.83,75.58){\line(1,0){19.67}}
\put(35.25,103.83){\line(0,0){0}}
\put(35,74.83){\line(0,0){0}}
\put(103,44.33){\line(0,0){0}}
\put(82.75,103.83){\line(0,0){0}}
\put(55.25,43.58){\line(0,0){0}}
\put(129.5,104.83){\line(0,0){0}}
\put(81.25,74.58){\line(0,0){0}}
\put(128.5,75.58){\line(0,0){0}}
\put(35.25,81.17){\line(0,0){0}}
\put(35,52.17){\line(0,0){0}}
\put(103,21.67){\line(0,0){0}}
\put(82.75,81.17){\line(0,0){0}}
\put(55.25,20.92){\line(0,0){0}}
\put(129.5,82.17){\line(0,0){0}}
\put(81.25,51.92){\line(0,0){0}}
\put(128.5,52.92){\line(0,0){0}}
\put(35,80.92){\line(-1,0){19.67}}
\put(34.75,51.92){\line(-1,0){19.67}}
\put(102.75,21.42){\line(-1,0){19.67}}
\put(82.5,80.92){\line(-1,0){19.67}}
\put(55,20.67){\line(-1,0){19.67}}
\put(129.25,81.92){\line(-1,0){19.67}}
\put(81,51.67){\line(-1,0){19.67}}
\put(128.25,52.67){\line(-1,0){19.67}}
\put(15.58,81.17){\line(0,0){0}}
\put(15.33,52.17){\line(0,0){0}}
\put(83.33,21.67){\line(0,0){0}}
\put(63.08,81.17){\line(0,0){0}}
\put(35.58,20.92){\line(0,0){0}}
\put(109.83,82.17){\line(0,0){0}}
\put(61.58,51.92){\line(0,0){0}}
\put(108.83,52.92){\line(0,0){0}}
\put(15.58,81.17){\line(-5,4){12.33}}
\put(15.33,52.17){\line(-5,4){12.33}}
\put(83.33,21.67){\line(-5,4){12.33}}
\put(63.08,81.17){\line(-5,4){12.33}}
\put(35.58,20.92){\line(-5,4){12.33}}
\put(109.83,82.17){\line(-5,4){12.33}}
\put(61.58,51.92){\line(-5,4){12.33}}
\put(108.83,52.92){\line(-5,4){12.33}}
\put(3.25,91.17){\line(0,0){0}}
\put(3,62.17){\line(0,0){0}}
\put(71,31.67){\line(0,0){0}}
\put(63.5,32.92){\line(0,0){0}}
\put(50.75,91.17){\line(0,0){0}}
\put(23.25,30.92){\line(0,0){0}}
\put(97.5,92.17){\line(0,0){0}}
\put(49.25,61.92){\line(0,0){0}}
\put(96.5,62.92){\line(0,0){0}}
\put(3.25,91.17){\line(3,-1){32}}
\put(50.75,91.17){\line(3,-1){32}}
\put(23.25,30.92){\line(3,-1){32}}
\put(97.5,92.17){\line(3,-1){32}}
\put(35.25,80.5){\line(0,0){0}}
\put(35,51.5){\line(0,0){0}}
\put(103,21){\line(0,0){0}}
\put(82.75,80.5){\line(0,0){0}}
\put(55.25,20.25){\line(0,0){0}}
\put(129.5,81.5){\line(0,0){0}}
\put(81.25,51.25){\line(0,0){0}}
\put(128.5,52.25){\line(0,0){0}}
\put(35,80.25){\line(1,1){11.33}}
\put(34.75,51.25){\line(1,1){11.33}}
\put(102.75,20.75){\line(1,1){11.33}}
\put(82.5,80.25){\line(1,1){11.33}}
\put(55,20){\line(1,1){11.33}}
\put(129.25,81.25){\line(1,1){11.33}}
\put(81,51){\line(1,1){11.33}}
\put(128.25,52){\line(1,1){11.33}}
\put(46.58,91.83){\line(0,0){0}}
\put(46.33,62.83){\line(0,0){0}}
\put(114.33,32.33){\line(0,0){0}}
\put(94.08,91.83){\line(0,0){0}}
\put(66.58,31.58){\line(0,0){0}}
\put(140.83,92.83){\line(0,0){0}}
\put(92.58,62.58){\line(0,0){0}}
\put(139.83,63.58){\line(0,0){0}}
\put(46.58,91.83){\line(-1,1){11.33}}
\put(46.33,62.83){\line(-1,1){11.33}}
\put(114.33,32.33){\line(-1,1){11.33}}
\put(94.08,91.83){\line(-1,1){11.33}}
\put(66.58,31.58){\line(-1,1){11.33}}
\put(140.83,92.83){\line(-1,1){11.33}}
\put(92.58,62.58){\line(-1,1){11.33}}
\put(139.83,63.58){\line(-1,1){11.33}}
\put(35.25,103.17){\line(0,0){0}}
\put(35,74.17){\line(0,0){0}}
\put(103,43.67){\line(0,0){0}}
\put(82.75,103.17){\line(0,0){0}}
\put(55.25,42.92){\line(0,0){0}}
\put(129.5,104.17){\line(0,0){0}}
\put(81.25,73.92){\line(0,0){0}}
\put(128.5,74.92){\line(0,0){0}}
%\emline(15,81.75)(46.25,91.5)
\multiput(15,81.75)(.1081314879,.0337370242){289}{\line(1,0){.1081314879}}
%\end
%\emline(14.75,52.75)(46,62.5)
\multiput(14.75,52.75)(.1081314879,.0337370242){289}{\line(1,0){.1081314879}}
%\end
%\emline(82.75,22.25)(114,32)
\multiput(82.75,22.25)(.1081314879,.0337370242){289}{\line(1,0){.1081314879}}
%\end
%\emline(15.5,103.25)(34.75,80.5)
\multiput(15.5,103.25)(.0337127846,-.0398423818){571}{\line(0,-1){.0398423818}}
%\end
%\emline(15.25,74.25)(34.5,51.5)
\multiput(15.25,74.25)(.0337127846,-.0398423818){571}{\line(0,-1){.0398423818}}
%\end
%\emline(83.25,43.75)(102.5,21)
\multiput(83.25,43.75)(.0337127846,-.0398423818){571}{\line(0,-1){.0398423818}}
%\end
%\emline(50.75,91.5)(82.5,103.5)
\multiput(50.75,91.5)(.0891853933,.0337078652){356}{\line(1,0){.0891853933}}
%\end
%\emline(63,103.25)(93.75,91.75)
\multiput(63,103.25)(.0901759531,-.0337243402){341}{\line(1,0){.0901759531}}
%\end
%\emline(109.75,104.25)(140.5,92.75)
\multiput(109.75,104.25)(.0901759531,-.0337243402){341}{\line(1,0){.0901759531}}
%\end
%\emline(93.75,91.75)(62,81)
\multiput(93.75,91.75)(-.0995297806,-.0336990596){319}{\line(-1,0){.0995297806}}
%\end
%\emline(66.25,31.5)(34.5,20.75)
\multiput(66.25,31.5)(-.0995297806,-.0336990596){319}{\line(-1,0){.0995297806}}
%\end
%\emline(140.5,92.75)(108.75,82)
\multiput(140.5,92.75)(-.0995297806,-.0336990596){319}{\line(-1,0){.0995297806}}
%\end
%\emline(96.25,62.75)(139.5,63.75)
\multiput(96.25,62.75)(1.4416667,.0333333){30}{\line(1,0){1.4416667}}
%\end
%\emline(128.25,74.25)(128,53)
\multiput(128.25,74.25)(-.03125,-2.65625){8}{\line(0,-1){2.65625}}
%\end
%\emline(128,53)(108.25,75)
\multiput(128,53)(-.0337030717,.0375426621){586}{\line(0,1){.0375426621}}
%\end
%\emline(128.5,75.5)(108,52.75)
\multiput(128.5,75.5)(-.0337171053,-.0374177632){608}{\line(0,-1){.0374177632}}
%\end
%\emline(3.5,91.25)(35.25,103.25)
\multiput(3.5,91.25)(.0891853933,.0337078652){356}{\line(1,0){.0891853933}}
%\end
%\emline(3.25,62.25)(35,74.25)
\multiput(3.25,62.25)(.0891853933,.0337078652){356}{\line(1,0){.0891853933}}
%\end
%\emline(71.25,31.75)(103,43.75)
\multiput(71.25,31.75)(.0891853933,.0337078652){356}{\line(1,0){.0891853933}}
%\end
%\emline(15.5,102.75)(46.25,91.5)
\multiput(15.5,102.75)(.0920658683,-.0336826347){334}{\line(1,0){.0920658683}}
%\end
%\emline(15.25,73.75)(46,62.5)
\multiput(15.25,73.75)(.0920658683,-.0336826347){334}{\line(1,0){.0920658683}}
%\end
%\emline(83.25,43.25)(114,32)
\multiput(83.25,43.25)(.0920658683,-.0336826347){334}{\line(1,0){.0920658683}}
%\end
%\emline(50.75,91.5)(93.75,91.75)
\multiput(50.75,91.5)(5.375,.03125){8}{\line(1,0){5.375}}
%\end
%\emline(97.5,92.5)(140.5,92.75)
\multiput(97.5,92.5)(5.375,.03125){8}{\line(1,0){5.375}}
%\end
%\emline(109,82)(129.5,104.5)
\multiput(109,82)(.0337171053,.0370065789){608}{\line(0,1){.0370065789}}
%\end
%\emline(3,62)(46.25,62.75)
\multiput(3,62)(1.8804348,.0326087){23}{\line(1,0){1.8804348}}
%\end
%\emline(15.5,74.25)(14.75,52.5)
\multiput(15.5,74.25)(-.0326087,-.9456522){23}{\line(0,-1){.9456522}}
%\end
%\emline(83.5,43.75)(82.75,22)
\multiput(83.5,43.75)(-.0326087,-.9456522){23}{\line(0,-1){.9456522}}
%\end
%\emline(62,73.25)(81,51.5)
\multiput(62,73.25)(.0336879433,-.0385638298){564}{\line(0,-1){.0385638298}}
%\end
%\emline(61.5,52)(81.25,74)
\multiput(61.5,52)(.0337030717,.0375426621){586}{\line(0,1){.0375426621}}
%\end
%\emline(49.25,62)(92.25,62.25)
\multiput(49.25,62)(5.375,.03125){8}{\line(1,0){5.375}}
%\end
%\emline(49,61.75)(80.75,73.75)
\multiput(49,61.75)(.0891853933,.0337078652){356}{\line(1,0){.0891853933}}
%\end
%\emline(49,61.5)(81.25,51.5)
\multiput(49,61.5)(.1085858586,-.0336700337){297}{\line(1,0){.1085858586}}
%\end
%\emline(108.5,75.5)(108.75,52.75)
\multiput(108.5,75.5)(.03125,-2.84375){8}{\line(0,-1){2.84375}}
%\end
%\emline(23,31.25)(55,43)
\multiput(23,31.25)(.0916905444,.0336676218){349}{\line(1,0){.0916905444}}
%\end
%\emline(55,43)(34.75,21)
\multiput(55,43)(-.0336938436,-.0366056572){601}{\line(0,-1){.0366056572}}
%\end
\put(119,78.5){\makebox(0,0)[cc]{$(iii)$}}
\put(22.5,48.75){\makebox(0,0)[cc]{$(iv)$}}
\put(68.75,47){\makebox(0,0)[cc]{$(v)$}}
\put(117.75,49){\makebox(0,0)[cc]{$(vi)$}}
\put(43,17){\makebox(0,0)[cc]{$(vii)$}}
\put(23.5,77.25){\makebox(0,0)[cc]{$(i)$}}
\put(71.25,78){\makebox(0,0)[cc]{$(ii)$}}
\put(63.25,7.75){\makebox(0,0)[cc]{\bf Figure 11}}
%\emline(35.25,43.5)(66.5,31.5)
\multiput(35.25,43.5)(.0877808989,-.0337078652){356}{\line(1,0){.0877808989}}
%\end
%\emline(70.75,32)(102.5,21.25)
\multiput(70.75,32)(.0995297806,-.0336990596){319}{\line(1,0){.0995297806}}
%\end
\put(91.25,18.5){\makebox(0,0)[cc]{$(viii)$}}
\end{picture}
 \\ There are 8 non isomorphic connected simple graphs, each with 6 vertices
 and 11 edges (\cite{Harary} pp. 223) as shown in Figure 11.
  It follows that by Lemma \ref{pirouz lemma}$(i)$  and $(iv)$ and Lemma \ref{pirouz lemma1} t
  hat $G$ can not be obtained from the graphs $(ii),(iii)$ and $(vii)$ of Figure 11.
  Suppose that $G$ is obtained from the graphs $(i)$ or $(iv)$. Then $G$ is isomorphic to
  one of the four graphs of Figure 12. By Lemma \ref{pirouz lemma}, $G$ is not isomorphic
  to each of the two graphs $(i)$ and $(ii)$ of Figure 12. Hence $G$ is isomorphic to
  graphs $(iii)$ and $(iv)$ of Figure 12, which are nothing but the graphs $G_{12}$ and $G_{13}$ of Figure 9.\\

  %TeXCAD Picture [fig 14.bak]. Options:
%\grade{\on}
%\emlines{\off}
%\epic{\off}
%\beziermacro{\on}
%\reduce{\on}
%\snapping{\off}
%\quality{8.000}
%\graddiff{0.005}
%\snapasp{1}
%\zoom{4.0000}
\unitlength 0.8mm % = 2.845pt
\linethickness{0.4pt}
\ifx\plotpoint\undefined\newsavebox{\plotpoint}\fi % GNUPLOT compatibility
\begin{picture}(120.625,88.875)(0,0)
\put(34.83,52.25){\circle*{1.33}}
\put(85.33,52.75){\circle*{1.33}}
\put(32.33,18.25){\circle*{1.33}}
\put(85.58,19.5){\circle*{1.33}}
\put(35.08,73.83){\circle*{1.33}}
\put(85.33,75.33){\circle*{1.33}}
\put(32.58,39.83){\circle*{1.33}}
\put(85.83,41.08){\circle*{1.33}}
\put(54.75,74.25){\circle*{1.33}}
\put(105.25,74.75){\circle*{1.33}}
\put(52.25,40.25){\circle*{1.33}}
\put(105.5,41.5){\circle*{1.33}}
\put(54.5,51.92){\circle*{1.33}}
\put(105,52.42){\circle*{1.33}}
\put(52,17.92){\circle*{1.33}}
\put(105.25,19.17){\circle*{1.33}}
\put(66.08,62.5){\circle*{1.33}}
\put(116.58,63){\circle*{1.33}}
\put(63.58,28.5){\circle*{1.33}}
\put(116.83,29.75){\circle*{1.33}}
\put(22.75,62.17){\circle*{1.33}}
\put(73.25,62.67){\circle*{1.33}}
\put(20.25,28.17){\circle*{1.33}}
\put(73.5,29.42){\circle*{1.33}}
\put(22.75,62.5){\line(1,1){12.33}}
\put(73.25,63){\line(1,1){12.33}}
\put(20.25,28.5){\line(1,1){12.33}}
\put(73.5,29.75){\line(1,1){12.33}}
\put(35.08,74.83){\line(0,0){0}}
\put(85.58,75.33){\line(0,0){0}}
\put(32.58,40.83){\line(0,0){0}}
\put(85.83,42.08){\line(0,0){0}}
\put(54.75,74.83){\line(0,0){0}}
\put(105.25,75.33){\line(0,0){0}}
\put(52.25,40.83){\line(0,0){0}}
\put(105.5,42.08){\line(0,0){0}}
\put(54.75,52.17){\line(0,0){0}}
\put(105.25,52.67){\line(0,0){0}}
\put(52.25,18.17){\line(0,0){0}}
\put(105.5,19.42){\line(0,0){0}}
\put(54.5,51.92){\line(-1,0){19.67}}
\put(105,52.42){\line(-1,0){19.67}}
\put(52,17.92){\line(-1,0){19.67}}
\put(105.25,19.17){\line(-1,0){19.67}}
\put(35.08,52.17){\line(0,0){0}}
\put(85.58,52.67){\line(0,0){0}}
\put(32.58,18.17){\line(0,0){0}}
\put(85.83,19.42){\line(0,0){0}}
\put(35.08,52.17){\line(-5,4){12.33}}
\put(85.58,52.67){\line(-5,4){12.33}}
\put(32.58,18.17){\line(-5,4){12.33}}
\put(85.83,19.42){\line(-5,4){12.33}}
\put(22.75,62.17){\line(0,0){0}}
\put(73.25,62.67){\line(0,0){0}}
\put(20.25,28.17){\line(0,0){0}}
\put(73.5,29.42){\line(0,0){0}}
\put(22.75,62.17){\line(3,-1){32}}
\put(73.25,62.67){\line(3,-1){32}}
\put(54.75,51.5){\line(0,0){0}}
\put(105.25,52){\line(0,0){0}}
\put(52.25,17.5){\line(0,0){0}}
\put(105.5,18.75){\line(0,0){0}}
\put(54.5,51.25){\line(1,1){11.33}}
\put(105,51.75){\line(1,1){11.33}}
\put(52,17.25){\line(1,1){11.33}}
\put(105.25,18.5){\line(1,1){11.33}}
\put(66.08,62.83){\line(0,0){0}}
\put(116.58,63.33){\line(0,0){0}}
\put(63.58,28.83){\line(0,0){0}}
\put(116.83,30.08){\line(0,0){0}}
\put(66.08,62.83){\line(-1,1){11.33}}
\put(116.58,63.33){\line(-1,1){11.33}}
\put(63.58,28.83){\line(-1,1){11.33}}
\put(116.83,30.08){\line(-1,1){11.33}}
\put(54.75,74.17){\line(0,0){0}}
\put(105.25,74.67){\line(0,0){0}}
\put(52.25,40.17){\line(0,0){0}}
\put(105.5,41.42){\line(0,0){0}}
%\emline(34.5,52.75)(65.75,62.5)
\multiput(34.5,52.75)(.1081314879,.0337370242){289}{\line(1,0){.1081314879}}
%\end
%\emline(35,74.25)(54.25,51.5)
\multiput(35,74.25)(.0337127846,-.0398423818){571}{\line(0,-1){.0398423818}}
%\end
%\emline(73.25,63)(105,75)
\multiput(73.25,63)(.0891853933,.0337078652){356}{\line(1,0){.0891853933}}
%\end
%\emline(85.5,74.75)(116.25,63.25)
\multiput(85.5,74.75)(.0901759531,-.0337243402){341}{\line(1,0){.0901759531}}
%\end
%\emline(116.25,63.25)(84.5,52.5)
\multiput(116.25,63.25)(-.0995297806,-.0336990596){319}{\line(-1,0){.0995297806}}
%\end
%\emline(32,18.5)(63.5,28.75)
\multiput(32,18.5)(.1036184211,.0337171053){304}{\line(1,0){.1036184211}}
%\end
%\emline(32,40.5)(63.25,28.75)
\multiput(32,40.5)(.0895415473,-.0336676218){349}{\line(1,0){.0895415473}}
%\end
%\emline(73.25,29.25)(116.5,30.25)
\multiput(73.25,29.25)(1.4416667,.0333333){30}{\line(1,0){1.4416667}}
%\end
%\emline(105,19.5)(85.25,41.5)
\multiput(105,19.5)(-.0337030717,.0375426621){586}{\line(0,1){.0375426621}}
%\end
%\emline(22.75,62)(54.75,74.25)
\multiput(22.75,62)(.0879120879,.0336538462){364}{\line(1,0){.0879120879}}
%\end
%\emline(35,74)(66,62.5)
\multiput(35,74)(.0909090909,-.0337243402){341}{\line(1,0){.0909090909}}
%\end
%\emline(66,62.5)(22.5,61.75)
\multiput(66,62.5)(-1.8913043,-.0326087){23}{\line(-1,0){1.8913043}}
%\end
\qbezier(54.5,74.5)(23.5,88.875)(22.5,61.75)
\qbezier(105.25,75.5)(120.625,76.125)(116.5,63.25)
\qbezier(52.25,40)(65.375,40.25)(63,29.5)
\qbezier(105.25,42.25)(72.75,52)(73.25,29.75)
%\emline(73.5,62.75)(116,63.25)
\multiput(73.5,62.75)(2.8333333,.0333333){15}{\line(1,0){2.8333333}}
%\end
%\emline(32.25,39.75)(32,18.25)
\multiput(32.25,39.75)(-.03125,-2.6875){8}{\line(0,-1){2.6875}}
%\end
%\emline(32.25,40)(52,17.75)
\multiput(32.25,40)(.0337030717,-.0379692833){586}{\line(0,-1){.0379692833}}
%\end
%\emline(51.75,39.75)(19.5,27.75)
\multiput(51.75,39.75)(-.0905898876,-.0337078652){356}{\line(-1,0){.0905898876}}
%\end
%\emline(20,28)(63.25,28.75)
\multiput(20,28)(1.8804348,.0326087){23}{\line(1,0){1.8804348}}
%\end
%\emline(73.25,29)(105.25,41.75)
\multiput(73.25,29)(.0846560847,.0337301587){378}{\line(1,0){.0846560847}}
%\end
%\emline(85.75,41)(116.75,30)
\multiput(85.75,41)(.0948012232,-.0336391437){327}{\line(1,0){.0948012232}}
%\end
%\emline(86,40.5)(85.5,18.75)
\multiput(86,40.5)(-.0333333,-1.45){15}{\line(0,-1){1.45}}
%\end
%\emline(85.25,19.5)(116.75,29.75)
\multiput(85.25,19.5)(.1036184211,.0337171053){304}{\line(1,0){.1036184211}}
%\end
\put(43.25,48.5){\makebox(0,0)[cc]{$(i)$}}
\put(95,49.25){\makebox(0,0)[cc]{$(ii)$}}
\put(41.5,13.25){\makebox(0,0)[cc]{$(iii)$}}
\put(95.5,15.5){\makebox(0,0)[cc]{$(iv)$}}
\put(63.75,7.25){\makebox(0,0)[cc]{\bf Figure 12}}
\end{picture}

  By Lemma \ref{pirouz lemma}, $G$ can not be obtained from graph $(v)$ of Figure 11. Suppose that
   $G$ is obtained from graph $(viii)$ of Figure 11. Then  $G$  is isomorphic to one of the two graphs
   of Figure 13. By Lemma \ref{pirouz lemma} $(iv)$ and \ref{pirouz lemma1}, $G$ is not isomorphic to graph $(i)$
   of Figure 13. By Lemma \ref{pirouz lemma} $(ii)$ and $(iv)$ and the fact that $M{(K_5)}$ does not contain odd
   cocircuit, $G$ can not be isomorphic to the graph $(ii)$ of Figure 13.\\
  %TeXCAD Picture [fig 15.bak]. Options:
%\grade{\on}
%\emlines{\off}
%\epic{\off}
%\beziermacro{\on}
%\reduce{\on}
%\snapping{\off}
%\quality{8.000}
%\graddiff{0.005}
%\snapasp{1}
%\zoom{4.0000}
\unitlength 0.8mm % = 2.845pt
\linethickness{0.4pt}
\ifx\plotpoint\undefined\newsavebox{\plotpoint}\fi % GNUPLOT compatibility
\begin{picture}(119.5,52.5)(0,0)
\put(32.33,18.25){\circle*{1.33}}
\put(85.58,19.5){\circle*{1.33}}
\put(32.58,39.83){\circle*{1.33}}
\put(85.83,41.08){\circle*{1.33}}
\put(52.25,40.25){\circle*{1.33}}
\put(105.5,41.5){\circle*{1.33}}
\put(52,17.92){\circle*{1.33}}
\put(105.25,19.17){\circle*{1.33}}
\put(63.58,28.5){\circle*{1.33}}
\put(116.83,29.75){\circle*{1.33}}
\put(20.25,28.17){\circle*{1.33}}
\put(73.5,29.42){\circle*{1.33}}
\put(20.25,28.5){\line(1,1){12.33}}
\put(73.5,29.75){\line(1,1){12.33}}
\put(32.58,40.83){\line(0,0){0}}
\put(85.83,42.08){\line(0,0){0}}
\put(52.25,40.83){\line(0,0){0}}
\put(105.5,42.08){\line(0,0){0}}
\put(52.25,18.17){\line(0,0){0}}
\put(105.5,19.42){\line(0,0){0}}
\put(52,17.92){\line(-1,0){19.67}}
\put(105.25,19.17){\line(-1,0){19.67}}
\put(32.58,18.17){\line(0,0){0}}
\put(85.83,19.42){\line(0,0){0}}
\put(32.58,18.17){\line(-5,4){12.33}}
\put(85.83,19.42){\line(-5,4){12.33}}
\put(20.25,28.17){\line(0,0){0}}
\put(73.5,29.42){\line(0,0){0}}
\put(52.25,17.5){\line(0,0){0}}
\put(105.5,18.75){\line(0,0){0}}
\put(52,17.25){\line(1,1){11.33}}
\put(105.25,18.5){\line(1,1){11.33}}
\put(63.58,28.83){\line(0,0){0}}
\put(116.83,30.08){\line(0,0){0}}
\put(63.58,28.83){\line(-1,1){11.33}}
\put(116.83,30.08){\line(-1,1){11.33}}
\put(52.25,40.17){\line(0,0){0}}
\put(105.5,41.42){\line(0,0){0}}
%\emline(32,18.5)(63.5,28.75)
\multiput(32,18.5)(.1036184211,.0337171053){304}{\line(1,0){.1036184211}}
%\end
%\emline(32,40.5)(63.25,28.75)
\multiput(32,40.5)(.0895415473,-.0336676218){349}{\line(1,0){.0895415473}}
%\end
%\emline(105,19.5)(85.25,41.5)
\multiput(105,19.5)(-.0337030717,.0375426621){586}{\line(0,1){.0375426621}}
%\end
%\emline(32.25,39.75)(32,18.25)
\multiput(32.25,39.75)(-.03125,-2.6875){8}{\line(0,-1){2.6875}}
%\end
%\emline(32.25,40)(52,17.75)
\multiput(32.25,40)(.0337030717,-.0379692833){586}{\line(0,-1){.0379692833}}
%\end
%\emline(51.75,39.75)(19.5,27.75)
\multiput(51.75,39.75)(-.0905898876,-.0337078652){356}{\line(-1,0){.0905898876}}
%\end
%\emline(73.25,29)(105.25,41.75)
\multiput(73.25,29)(.0846560847,.0337301587){378}{\line(1,0){.0846560847}}
%\end
%\emline(85.75,41)(116.75,30)
\multiput(85.75,41)(.0948012232,-.0336391437){327}{\line(1,0){.0948012232}}
%\end
%\emline(85.25,19.5)(116.75,29.75)
\multiput(85.25,19.5)(.1036184211,.0337171053){304}{\line(1,0){.1036184211}}
%\end
\qbezier(52,40.5)(29.625,52.5)(19.75,28.5)
\qbezier(116.75,30)(119.5,43)(105.25,42)
%\emline(73.25,29.25)(105,19)
\multiput(73.25,29.25)(.1044407895,-.0337171053){304}{\line(1,0){.1044407895}}
%\end
\put(41.25,13.5){\makebox(0,0)[cc]{$(i)$}}
\put(93.75,15.75){\makebox(0,0)[cc]{$(ii)$}}
\put(61.25,7.5){\makebox(0,0)[cc]{\bf Figure 13}}
\end{picture}

  Suppose that $G$ is obtained from graph $(vi)$   of Figure 11.
  Then $G$ is isomorphic to one of the two graphs $G_{14}$ and $G_{15}$ of Figure 9.
\vskip.2cm\noindent {\bf Case (iii)}.
$M'_{x,y}\cong M(K_{5}).$\vskip.2cm\noindent Then a contradiction
follows from Lemma \ref{8 properties} $(viii).$

\vskip.2cm\noindent {\bf Case(iv)}. $M'_{x,y}/\{x\}\cong M(K_{5}).$\vskip.2cm\noindent {\bf
Subcase (i)}. Suppose that $M$ is not cographic.\vskip.2cm\noindent Let
$G$ be a graph that corresponds to the matroid $M.$ Since
$r(M(K_{5}))= 4,~r(M'_{x,y})=5.$ Further, $r(M)=4$ and
$|E(M)|=10.$ Therefore, $G$ has 5 vertihapterces and 10 edges. $M$ has no minor isomorphic to $M(K_{3,3})$ as $K_{3,3}$
has 6 vertices. Suppose that $M$ has a minor isomorphic to $M(K_{5})$ then $M\cong M(K_5)$.
By Lemma \ref{8 properties}, $x,y$ can not both be in a triangle. Hence $x$ and $y$ are not adjacent.
 Let $C^*$ ba cocircuit of $M$ containing $y$ but not $x$ such that $|C^*|=4$, since we can always find a set
  of 4 edges containing $y$ incident to some vertex in  $K_5$, that set of edges is a cocircuit of $M(K_{5})$. Then $C^*\cup \{a\}$ becomes a cocircuit
  of $M'_{x,y}/\{x\}\cong M(K_{5})$, a contradiction to the fact that $M(K_{5})$ is Eulerian and $|C^*\cup \{a\}|=5.$
 Thus $M$ is cographic.\\
Since $M$ is graphic and cographic, $G$ is planar. By Harary
\cite{Harary}, there is no simple planar graph with 5 vertices and 10
edges. Hence $G$ must be non-simple. Then, by Lemma \ref{8
properties}$(vi),~G$ has exactly one pair of parallel edges. $G$
can be obtained from a simple planar graph with 5 vertices and 9
edges by adding an edge in parallel. By Harary \cite{Harary}, every
simple planar graph with 5 vertices and 9 edges is isomorphic to
graph $(i)$ of Figure 14. Therefore, $G$ is isomorphic to the
graph $(ii)$ or $(iii)$ of Figure 14. In graph $(iii),$ there are
two edge-disjoint 3-cutsets (i.e. 3-cocircuits in $M(G)$). Hence,
one of them is preserved in $M'_{x,y}/\{x\},$ and this is a
contradiction. Thus, $G$ is isomorphic to graph $(ii)$ of Figure
16, which is nothing but the
graph $G_{16}$ of Figure 9. \\\\
\unitlength=0.8mm \special{em:linewidth 0.4pt} \linethickness{0.4pt}
\begin{picture}(105.33,35.33)
\put(45.00,15.00){\circle*{1.33}}
\put(59.00,15.00){\circle*{1.33}}
\put(59.00,26.33){\circle*{1.33}}
\put(45.00,26.33){\circle*{1.33}}
\put(52.33,34.00){\circle*{1.33}}
\put(67.67,26.33){\circle*{1.33}}
\put(67.67,15.00){\circle*{1.33}}
\put(81.67,15.00){\circle*{1.33}}
\put(81.67,26.33){\circle*{1.33}}
\put(74.67,34.00){\circle*{1.33}}
\put(90.33,26.33){\circle*{1.33}}
\put(104.67,26.33){\circle*{1.33}}
\put(104.67,15.00){\circle*{1.33}}
\put(90.33,15.00){\circle*{1.33}}
\put(97.67,34.00){\circle*{1.33}}
\put(52.33,34.00){\line(5,-6){6.33}}
\put(59.00,26.33){\line(0,-1){11.33}}
\put(59.00,15.00){\line(-1,0){14.00}}
\put(45.00,15.00){\line(0,1){11.33}}
\put(45.00,26.33){\line(1,1){7.33}}
\put(52.33,34.00){\line(-2,-5){7.67}}
\put(45.00,15.00){\line(5,4){14.00}}
\put(45.00,26.33){\line(5,-4){14.00}}
\put(59.00,15.00){\line(-1,3){6.33}}
\put(67.67,26.33){\line(5,6){6.33}}
\put(74.67,34.00){\line(5,-6){6.33}}
\put(81.67,26.33){\line(0,-1){11.33}}
\put(81.67,15.00){\line(-1,0){14.00}}
\put(67.67,15.00){\line(0,1){11.33}}
\put(67.67,26.33){\line(5,-4){14.00}}
\put(81.67,15.00){\line(-1,3){6.33}}
\put(74.67,34.33){\line(-1,-3){6.33}}
\put(67.67,15.00){\line(5,4){14.00}}
\put(90.33,26.33){\line(1,1){7.33}}
\put(97.67,34.00){\line(5,-6){6.33}}
\put(104.67,26.33){\line(0,-1){11.33}}
\put(104.67,15.00){\line(-1,0){14.33}}
\put(90.33,15.00){\line(0,1){11.33}}
\put(90.33,26.33){\line(5,-4){14.33}}
\put(104.67,15.00){\line(-1,3){6.33}}
\put(97.67,34.33){\line(-2,-5){7.67}}
\put(90.33,15.00){\line(5,4){14.33}}
\bezier{84}(104.67,15.00)(97.67,7.00)(90.33,15.00)
\bezier{80}(74.67,34.00)(64.00,35.33)(67.67,26.33)
\put(52.33,7.33){\makebox(0,0)[cc]{$(i)$}}
\put(74.67,7.33){\makebox(0,0)[cc]{$(ii)$}}
\put(97.67,7.33){\makebox(0,0)[cc]{$(iii)$}}
\put(74.67,0.00){\makebox(0,0)[cc]{\bf Figure 14}}
\end{picture}
\\
\vskip.2cm\noindent
\vskip.2cm\noindent {\bf
Case (v)}. $M'_{x,y}/\{x,y\}\cong M(K_{5}).$\vskip.2cm\noindent
{\bf Subcase (i)}. Suppose that $M$ is not cographic.\vskip.2cm
\noindent Let $G$ be a graph which corresponds to the matroid $M.$
Since $r(M(K_{5}))= 4,~r(M'_{x,y})=6.$ So, $r(M)=5.$ Further, $|E(M)|=11.$ \\
Suppose that $M$ has a minor isomorphic to $M(K_{3,3})$. Let $G$ be the connected
graph corresponding to $M$. Then $G$ is the graph with 6 vertices and 11 edges.
By Lemma \ref{8 properties} $(vii)$, $G$ has to be simple. Consequently, $G$ is isomorphic
to one of the  following two graphs in Figure 15 (see Harary \cite{Harary}).\\
\\
%TeXCAD Picture [fig 17.bak]. Options:
%\grade{\on}
%\emlines{\off}
%\epic{\off}
%\beziermacro{\on}
%\reduce{\on}
%\snapping{\off}
%\quality{8.000}
%\graddiff{0.005}
%\snapasp{1}
%\zoom{4.0000}
\unitlength 0.8mm % = 2.845pt
\linethickness{0.4pt}
\ifx\plotpoint\undefined\newsavebox{\plotpoint}\fi % GNUPLOT compatibility
\begin{picture}(108.25,38.75)(0,0)
\put(14,14.33){\circle*{1.33}}
\put(68.75,14.08){\circle*{1.33}}
\put(14,36.67){\circle*{1.33}}
\put(68.75,36.42){\circle*{1.33}}
\put(31.67,36){\circle*{1.33}}
\put(86.42,35.75){\circle*{1.33}}
\put(50.67,36){\circle*{1.33}}
\put(105.42,35.75){\circle*{1.33}}
\put(50.67,14){\circle*{1.33}}
\put(105.42,13.75){\circle*{1.33}}
\put(31.67,14){\circle*{1.33}}
\put(86.42,13.75){\circle*{1.33}}
%\emline(13.75,36.5)(33,35.75)
\multiput(13.75,36.5)(.8369565,-.0326087){23}{\line(1,0){.8369565}}
%\end
%\emline(68.5,36.25)(87.75,35.5)
\multiput(68.5,36.25)(.8369565,-.0326087){23}{\line(1,0){.8369565}}
%\end
\put(33,35.75){\line(0,1){0}}
\put(87.75,35.5){\line(0,1){0}}
%\emline(33,35.75)(50.5,36)
\multiput(33,35.75)(2.1875,.03125){8}{\line(1,0){2.1875}}
%\end
%\emline(50.5,36)(50.75,35.75)
\multiput(50.5,36)(.03125,-.03125){8}{\line(0,-1){.03125}}
%\end
%\emline(105.25,35.75)(105.5,35.5)
\multiput(105.25,35.75)(.03125,-.03125){8}{\line(0,-1){.03125}}
%\end
%\emline(14,36.75)(13.75,14.5)
\multiput(14,36.75)(-.03125,-2.78125){8}{\line(0,-1){2.78125}}
%\end
%\emline(68.75,36.5)(68.5,14.25)
\multiput(68.75,36.5)(-.03125,-2.78125){8}{\line(0,-1){2.78125}}
%\end
%\emline(13.75,14.5)(31.75,36.25)
\multiput(13.75,14.5)(.0337078652,.0407303371){534}{\line(0,1){.0407303371}}
%\end
%\emline(68.5,14.25)(86.5,36)
\multiput(68.5,14.25)(.0337078652,.0407303371){534}{\line(0,1){.0407303371}}
%\end
%\emline(31.75,36.25)(50.5,14.25)
\multiput(31.75,36.25)(.0337230216,-.0395683453){556}{\line(0,-1){.0395683453}}
%\end
%\emline(86.5,36)(105.25,14)
\multiput(86.5,36)(.0337230216,-.0395683453){556}{\line(0,-1){.0395683453}}
%\end
%\emline(50.5,14.25)(50.75,36.25)
\multiput(50.5,14.25)(.03125,2.75){8}{\line(0,1){2.75}}
%\end
%\emline(105.25,14)(105.5,36)
\multiput(105.25,14)(.03125,2.75){8}{\line(0,1){2.75}}
%\end
%\emline(50.75,36.25)(31.5,14.25)
\multiput(50.75,36.25)(-.0337127846,-.0385288967){571}{\line(0,-1){.0385288967}}
%\end
%\emline(105.5,36)(86.25,14)
\multiput(105.5,36)(-.0337127846,-.0385288967){571}{\line(0,-1){.0385288967}}
%\end
%\emline(31.5,14.25)(13.75,36.5)
\multiput(31.5,14.25)(-.0336812144,.0422201139){527}{\line(0,1){.0422201139}}
%\end
%\emline(86.25,14)(68.5,36.25)
\multiput(86.25,14)(-.0336812144,.0422201139){527}{\line(0,1){.0422201139}}
%\end
\put(31.5,35.5){\line(0,-1){21}}
\put(86.25,35.25){\line(0,-1){21}}
%\emline(13.75,14.75)(50.75,36)
\multiput(13.75,14.75)(.0587301587,.0337301587){630}{\line(1,0){.0587301587}}
%\end
%\emline(68.5,14.5)(105.5,35.75)
\multiput(68.5,14.5)(.0587301587,.0337301587){630}{\line(1,0){.0587301587}}
%\end
%\emline(50.25,14.25)(13.25,36.75)
\multiput(50.25,14.25)(-.05547226387,.03373313343){667}{\line(-1,0){.05547226387}}
%\end
%\emline(105,14)(68,36.5)
\multiput(105,14)(-.05547226387,.03373313343){667}{\line(-1,0){.05547226387}}
%\end
%\emline(68.75,14)(86.75,13.75)
\multiput(68.75,14)(2.25,-.03125){8}{\line(1,0){2.25}}
%\end
%\emline(86.75,13.75)(86.5,14)
\multiput(86.75,13.75)(-.03125,.03125){8}{\line(0,1){.03125}}
%\end
\put(30.75,8.5){\makebox(0,0)[cc]{$(i)$}}
\put(81.5,9.5){\makebox(0,0)[cc]{}}
\put(86.5,9.75){\makebox(0,0)[cc]{$(ii)$}}
\put(76.25,38.75){\makebox(0,0)[cc]{$x$}}
\put(108.25,25.75){\makebox(0,0)[cc]{$y$}}
\put(53.75,3){\makebox(0,0)[cc]{\bf Figure 15}}
\end{picture}

In the graph $(i)$ of Figure 15, any two edges are either in a 3-cycle
or a 4-cycle. The elements $x$ and $y$ can not be in a 3-circuit or a 4-circuit because such circuit
becomes a loop or a 2-circuit in $M'_{x,y}/\{x,y\}\cong M(K_{5}).$ This is a contradiction
to the fact that $M(K_5)$ is simple.\\
In the graph $(ii)$  of Figure 15, $M'_{x,y}/\{x,y\}\cong M(K_{5})$
for the edges $x, y$ shown in the graph $(ii)$ of Figure 15. Hence the graph $(ii)$ which is $G_{17}$ of  Figure 9, is a minimal graph
with respect to  $M(K_{5}).$ Suppose that $M$ has a minor
isomorphic to $M(K_{5})$. Since $r(M(K_{5}))= 4,~r(M'_{x,y})=6.$
So, $r(M)=5.$ Further, $|E(M)|=11.$ Let $G$ be a graph corresponding
 to the matroid $M.$ Thus $G$ is the graph with 6 vertices
 and 11 edges with $M(K_{5})$ as a minor. In fact $G$ is the graph shown in Figure 16.\\
\\
 %TeXCAD Picture [fig 18.bak]. Options:
%\grade{\on}
%\emlines{\off}
%\epic{\off}
%\beziermacro{\on}
%\reduce{\on}
%\snapping{\off}
%\quality{8.000}
%\graddiff{0.005}
%\snapasp{1}
%\zoom{4.0000}
\unitlength 0.8mm % = 2.845pt
\linethickness{0.4pt}
\ifx\plotpoint\undefined\newsavebox{\plotpoint}\fi % GNUPLOT compatibility
\begin{picture}(108.995,41.415)(0,0)
\put(77.08,18.75){\circle*{1.33}}
\put(77.33,40.33){\circle*{1.33}}
\put(97,40.75){\circle*{1.33}}
\put(96.75,18.42){\circle*{1.33}}
\put(108.33,29){\circle*{1.33}}
\put(65,28.67){\circle*{1.33}}
\put(65,29){\line(1,1){12.33}}
\put(77.33,41.33){\line(0,0){0}}
\put(97,41.33){\line(0,0){0}}
\put(97,18.67){\line(0,0){0}}
\put(96.75,18.42){\line(-1,0){19.67}}
\put(77.33,18.67){\line(0,0){0}}
\put(77.33,18.67){\line(-5,4){12.33}}
\put(65,28.67){\line(0,0){0}}
\put(97,18){\line(0,0){0}}
\put(96.75,17.75){\line(1,1){11.33}}
\put(108.33,29.33){\line(0,0){0}}
\put(108.33,29.33){\line(-1,1){11.33}}
\put(97,40.67){\line(0,0){0}}
%\emline(76.75,19)(108.25,29.25)
\multiput(76.75,19)(.1036184211,.0337171053){304}{\line(1,0){.1036184211}}
%\end
%\emline(77,40.25)(76.75,18.75)
\multiput(77,40.25)(-.03125,-2.6875){8}{\line(0,-1){2.6875}}
%\end
%\emline(77,40.5)(96.75,18.25)
\multiput(77,40.5)(.0337030717,-.0379692833){586}{\line(0,-1){.0379692833}}
%\end
%\emline(77.25,41)(97.25,40.75)
\multiput(77.25,41)(2.5,-.03125){8}{\line(1,0){2.5}}
%\end
\put(64.25,29.25){\line(1,0){43}}
\put(96,18.5){\line(-3,1){31.5}}
\put(84.5,9.25){\makebox(0,0)[cc]{\bf Figure 16}}
\end{picture}
\\
Observe that the graph of Figure 16, contains a vertex of degree 2 i.e. there is a 2-cocircuit in $M$.
This is a contradiction to the fact that $M$ does not have any pair of elements in
 series (i.e. 2-cocircuit). Hence $M$ has no minor isomorphic to $M(K_{5})$.\\
Thus $M$ is cographic. Consequently, $G$ is the planar graph with 6 vertices and 11
edges and has minimum degree at least 3. By Lemma \ref{8 properties} $(vii),~G$
is simple.

There are in all 9 non-isomorphic simple graphs with 6 vertices
and 11 edges (see \cite{Harary}). Out of which, four graphs are non-planar
and two graphs contain a degree two vertex, remaining 3 graphs are
shown in Figure 17. Here, $G$ cannot have a 3 or a 4-cycle
containing both $x$ and $y.$
Also, each\\
\unitlength=0.8mm \special{em:linewidth 0.4pt} \linethickness{0.4pt}
\begin{picture}(116.33,34.67)
\put(30.00,15.00){\circle*{1.33}}
\put(43.33,15.00){\circle*{1.33}}
\put(24.00,23.67){\circle*{1.33}}
\put(30.00,32.67){\circle*{1.33}}
\put(43.33,32.67){\circle*{1.33}}
\put(49.67,23.67){\circle*{1.33}}
\put(58.00,23.67){\circle*{1.33}}
\put(63.67,32.67){\circle*{1.33}}
\put(76.00,32.67){\circle*{1.33}}
\put(82.33,23.67){\circle*{1.33}}
\put(76.00,15.00){\circle*{1.33}}
\put(63.67,15.00){\circle*{1.33}}
\put(91.33,23.67){\circle*{1.33}}
\put(97.00,32.67){\circle*{1.33}}
\put(109.67,32.67){\circle*{1.33}}
\put(115.67,23.67){\circle*{1.33}}
\put(109.67,15.00){\circle*{1.33}}
\put(97.00,15.00){\circle*{1.33}}
\put(24.00,23.67){\line(2,3){6.00}}
\put(30.00,32.67){\line(1,0){13.33}}
\put(43.33,32.67){\line(2,-3){6.00}}
\put(49.67,23.67){\line(-3,-4){6.33}}
\put(43.33,15.00){\line(-1,0){13.33}}
\put(30.00,15.00){\line(-2,3){5.67}}
\put(24.00,23.67){\line(1,0){25.67}}
\put(49.67,23.67){\line(-2,1){19.67}}
\put(43.33,32.67){\line(-2,-1){19.33}}
\put(24.00,23.67){\line(5,-2){19.33}}
\put(30.00,15.00){\line(5,2){19.67}}
\put(58.00,23.67){\line(3,5){5.33}}
\put(63.67,32.67){\line(1,0){12.33}}
\put(76.00,32.67){\line(2,-3){6.00}}
\put(82.33,23.67){\line(-3,-4){6.33}}
\put(76.00,15.00){\line(-1,0){12.33}}
\put(63.67,15.00){\line(-2,3){5.67}}
\put(58.00,23.67){\line(1,0){24.33}}
\put(82.33,23.67){\line(-2,-1){18.67}}
\put(63.67,15.00){\line(2,3){11.67}}
\put(82.67,23.67){\line(-2,1){19.00}}
\put(58.00,23.67){\line(2,-1){18.00}}
\put(91.33,23.67){\line(3,5){5.33}}
\put(97.00,32.67){\line(1,0){12.67}}
\put(109.67,32.67){\line(2,-3){6.00}}
\put(115.67,23.67){\line(-2,-3){5.67}}
\put(109.67,15.00){\line(-1,0){12.67}}
\put(97.00,15.00){\line(-2,3){5.67}}
\put(91.33,23.67){\line(2,1){18.33}}
\put(109.67,32.67){\line(-3,-4){13.33}}
\put(97.00,15.00){\line(2,1){18.67}}
\put(109.67,15.00){\line(-2,1){18.33}}
\put(97.00,32.67){\line(2,-1){18.67}}
\put(69.67,34.67){\makebox(0,0)[cc]{$k$}}
\put(80.67,18.33){\makebox(0,0)[cc]{$l$}}
\put(37.00,10.33){\makebox(0,0)[cc]{$(i)$}}
\put(70.33,10.33){\makebox(0,0)[cc]{$(ii)$}}
\put(103.67,10.33){\makebox(0,0)[cc]{$(iii)$}}
\put(70.33,2.33){\makebox(0,0)[cc]{\bf Figure 17}}
\end{picture}
\vskip.2cm\noindent
3-cocycle and a 5-cocycle of $G$ must
contain $x$ or $y.$ These conditions are not satisfied for the graphs
$(i)$ and $(iii)$ for any pair of edges $x,y.$ The choice for
$(x,~y)$ in graph $(ii)$ is $(k,~l).$ But then $M'_{x,y}/\{x,y\}$ is
not eulerian. Hence the cycle matroids of the graphs in Figure 17 are not minimal with respect to
$M(K_{5}).$\end{proof}
Now, we use Lemmas \ref{3.1}, \ref{3.2}, \ref{3.3} and \ref{3.4} to prove Theorem \ref{gcmain}.\vskip.2cm\noindent
{\bf Proof of Theorem \ref{gcmain}}. Let $M$ be a graphic matroid. On
combining Corollary \ref{minimal} and Lemmas \ref{3.1}, \ref{3.2}, \ref{3.3} and \ref{3.4}, it
follows that $M_{x,y}$ is cographic for every pair $\{x,y\}$ of
elements of $M$ if and only if $M$ has no minor isomorphic to any
of the matroids $M(G_{i}),~i=1,2,3,4,5,6,7,8,9,10,11,12,13,14,15,16,17$,
where the graphs $G_{i}$ are shown in the statements of the Lemmas \ref{3.1}, \ref{3.2}, \ref{3.3} and \ref{3.4}.
 However, one can check that each of the
 matroids $M(G_{i}),~i=1,2,3,4,6,7,8,9,10,11,12,13,14,15,16,17$ has matroid $M(G_{5})$ of Figure 5 as a minor.
Hence the proof.\\

\noindent {\bf Disclosures:} \\
{\bf Funding:} No funding is received for this study.\\
{\bf Conflict of Interest:} The authors declare that they have no conflict of interest.\\
{\bf Ethical approval:} This article does not contain any studies with animals performed by any of the authors.\\
{\bf Ethical approval:} This article does not contain any studies with human participants or animals performed by any of the authors.\\

\bibliographystyle{amsplain}
%\bibliography{xbib}

\end{document}